\documentclass[12pt]{amsart}

\usepackage{tikz}
\usetikzlibrary{calc,matrix, arrows, babel, patterns}
\usepackage{tikz-cd}

\usepackage[utf8]{inputenc}
\usepackage[T1]{fontenc}

\usepackage{amsmath,amsthm,amsfonts,amssymb,latexsym}

\usepackage{hyperref}

\usepackage{tabto}

\usepackage{cleveref}

\usepackage{color, xcolor}

\usepackage{enumerate, enumitem, verbatim, graphicx}

\usepackage{mathrsfs, thmtools}

\usepackage{setspace}

\DeclareMathOperator{\Q}{\mathbb{Q}}

\DeclareMathOperator{\ZB}{\mathbf{Z}}
\DeclareMathOperator{\CB}{\mathbf{C}}
\DeclareMathOperator{\OB}{\mathbf{O}}
\DeclareMathOperator{\NB}{\mathbf{N}}

\DeclareMathOperator{\Syl}{Syl}

\DeclareMathOperator{\Sub}{\mathbf{Sub}}

\DeclareMathOperator{\Irr}{Irr}

\theoremstyle{remark}
\newtheorem*{remark}{Remark}

\theoremstyle{definition}

\theoremstyle{plain}

\newtheorem{thm}{Theorem}[section]
\newtheorem{lem}[thm]{Lemma}

\newtheorem{pro}[thm]{Proposition}
\newtheorem{cor}[thm]{Corollary}

\newtheorem*{conUn}{Conjecture}

\newtheorem*{thmA}{Theorem A}
\newtheorem*{thmB}{Theorem B}
\newtheorem*{thmC}{Theorem C}

\theoremstyle{definition}

\newtheorem*{ack}{Acknowledgments}

\numberwithin{equation}{section}

\begin{document}

% Title
%%%%%%%%%%%%%%%%%%%%%%%%%%%%%%%%%%%%%%%%%%%%%%%%%%%%%%%%%%%%%
\title[Subnormalizers and character correspondences in $p$-solvable groups]{Subnormalizers and character correspondences in $p$-solvable groups}

% Authors Information
%%%%%%%%%%%%%%%%%%%%%%%%%%%%%%%%%%%%%%%%%%%%%%%%%%%%%%%%%%%%%

\author{Gabriel A. L. Souza}
\address{Departament de Matemàtiques, Universitat de València, 46100 Burjassot, València, Spain}
\email{gabriel.area@uv.es}

% Thank you and classification
%%%%%%%%%%%%%%%%%%%%%%%%%%%%%%%%%%%%%%%%%%%%%%%%%%%%%%%%%%%%%

\thanks{The author is supported by Ministerio de Ciencia e Innovación (grant PREP2022-000021 tied to the project PID2022-137612NB-I00 funded by MCIN/AEI/10.13039/501100011033 and ``ERDF A way of making Europe'').}



\date{\today}

% Abstract
%%%%%%%%%%%%%%%%%%%%%%%%%%%%%%%%%%%%%%%%%%%%%%%%%%%%%%%%%%%%%

\begin{abstract}
	A new family of local-global conjectures in the representation theory of finite groups has recently been proposed by Moretó. We show that one of the strongest of these conjectures, the strong subnormalizer conjecture, holds for $p$-solvable groups when $p$ is odd, under the condition that the subnormalizer subset is a subgroup. We also prove it in general when $p$ is odd and the $p$-length of the group is $1$ and, in the process, obtain new properties related to the Glauberman correspondence.
\end{abstract}

\maketitle

 % Main body
%%%%%%%%%%%%%%%%%%%%%%%%%%%%%%%%%%%%%%%%%%%%%%%%%%%%%%%%%%%%%

\section{Introduction}

Recently, Moretó \cite{Mor} has proposed a new family of conjectures which provide a different
framework for viewing some local-global problems in the representation theory of finite groups. Let $G$ be a finite group, let $p$ be a prime and let $x \in G$ be a $p$-element. We write $\Irr^x(G)$ for the set of irreducible complex characters of $G$ which are non-zero when evaluated at $x$. We call a $p$-element \emph{picky} if it lies in a unique Sylow $p$-subgroup. One of the formulations of one of the main proposed conjectures (see \cite[Conjecture A]{Mor}) is as follows:
\begin{conUn}[Picky conjecture]
	Let $G$ be a finite group, let $p$ be a prime and let $x \in G$ be a picky $p$-element. Let $P \in \Syl_p(G)$ contain $x$. Then, there exists a bijection
	\begin{equation*}
		f: \Irr^x(G) \to \Irr^x(\NB_G(P))
	\end{equation*}
	such that, for all $\chi \in \Irr^x(G)$, $\chi(1)_p = f(\chi)(1)_p$ and $\Q(f(\chi)(x)) = \Q(\chi(x))$.
\end{conUn}
In fact, for many classes of groups, we can actually find a bijection $f$ as above satisfying the stronger condition $f(\chi)(x) = \epsilon_\chi \chi(x)$ for a sign $\epsilon_\chi \in \{1, -1\}$. In this case, we say $G$ satisfies the \emph{strong picky conjecture}.

This conjecture (and its strong form) have started receiving substantial attention in recent months, being proved for quasisimple groups of Lie type in non-defining characteristic \cite{GunterMandiPicky} and in its strong form both for symmetric groups \cite{SymmetricPicky} and, even more recently, for $p$-solvable groups with $p \neq 2$ \cite{PickyPSolvable}.

There is, in fact, a generalization of the picky conjecture, which also appears in \cite{Mor} and was studied in \cite{GunterSubnormalizer} and \cite{SymmetricPicky}: the \emph{subnormalizer conjecture}. If $x \in G$ is a $p$-element, we define the \emph{subnormalizer subset} of $x$ in $G$ as $S_G(x) = \{y \in G \mid \langle x \rangle {\triangleleft}{\triangleleft} \langle x, y \rangle\}$. This is, in general, not a subgroup of $G$, and so we define the \emph{subnormalizer} of $x$ in $G$ as $\Sub_G(x) = \langle S_G(x) \rangle$. It is the case that a $p$-element $x \in G$ is picky if and only if $\Sub_G(x) = \NB_G(P)$, where $x \in P$ and $P \in \Syl_p(G)$. This suggested that the following conjecture, \cite[Conjecture B]{Mor}, could be true:
\begin{conUn}[Subnormalizer conjecture]
	Let $G$ be a finite group, let $p$ be a prime and let $x \in G$ be a $p$-element. Then, there exists a bijection
	\begin{equation*}
		f: \Irr^x(G) \to \Irr^x(\Sub_G(x))
	\end{equation*}
	such that, for all $\chi \in \Irr^x(G)$, $\chi(1)_p = f(\chi)(1)_p$ and $\Q(f(\chi)(x)) = \Q(\chi(x))$.
\end{conUn}
As was the case with the picky conjecture, for many groups, such an $f$ can be found satisfying the stronger condition $f(\chi)(x) = \epsilon_\chi \chi(x)$ for some sign $\epsilon_{\chi} \in \{1, -1\}$, in which case we say that $G$ satisfies the \emph{strong subnormalizer conjecture}. 

This conjecture appears to be very deep, and there are very few classes of groups for which it is known to hold. By work done in \cite{SymmetricPicky}, it holds for symmetric groups for $p = 2$; it is also known for some elements in groups of Lie type of small rank \cite{GunterSubnormalizer} or with abelian Sylow $p$-subgroups \cite{GunterAbelianSylow}. 

It is on this conjecture that we focus our efforts, in the case of $p$-solvable groups for $p \neq 2$. Our main result is as follows:

\begin{thmA}
	Let $G$ be a $p$-solvable group, for $p \neq 2$, and let $x \in G$ be a $p$-element. Assume that $\langle x \rangle {\triangleleft}{\triangleleft} \Sub_G(x)$. Then, there exist a bijection
	\begin{equation*}
		f : \Irr^x(G) \to \Irr^x(\Sub_G(x))
	\end{equation*}
	such that $f(\chi)(1)_p = \chi(1)_p$ and $f(\chi)(x) = \epsilon_\chi \chi(x)$, for some sign $\epsilon_\chi \in \{1, -1\}$ 
\end{thmA}

The condition $\langle x \rangle {\triangleleft}{\triangleleft} \Sub_G(x)$ is satisfied precisely when $S_G(x) = \Sub_G(x)$; that is, when $S_G(x)$ is a subgroup of $G$. This is because, by definition, $S_G(x)$ contains all subgroups in which $\langle x \rangle$ is subnormal. Notice how the strong picky conjecture for $p$-solvable groups with $p \neq 2$ follows from Theorem A, since $\langle x \rangle$ is always subnormal in the normalizer of a Sylow $p$-subgroup containing it. We also provide a couple of constructions to obtain elements satisfying this condition. Theorem A is thus a substantial generalization of \cite[Theorem A]{PickyPSolvable}.

Our general strategy for proving Theorem A follows that of \cite{PickyPSolvable}. However, most of the proof must employ significantly new ideas, adapted to this more general setting. 

For $p$-solvable groups of $p$-length one -- and, in particular, those with a normal $p$-complement -- we are able to go a step further and completely prove the strong subnormalizer conjecture for $p \neq 2$, under no hypotheses on $x$.

\begin{thmB}
	Let $p$ be an odd prime, let $G$ be a $p$-solvable group of $p$-length one and let $x$ be a $p$-element of $G$. Then, there exists a bijection $$f: \Irr^x(G) \to \Irr^x(\Sub_G(x))$$ such that $f(\chi)(1)_p = \chi(1)_p$ and $f(\chi)(x) = \epsilon_\chi \chi(x)$ for some $\epsilon_\chi \in \{1, -1\}$.
\end{thmB}

In the process of proving Theorem B, we obtain the structure of $\Sub_G(x)$ when $G$ has a normal $p$-complement, as well as when $G$ is $p$-solvable with $p$-length one, and show that it is relatively controlled.

In order to prove Theorems A and B, we will need to study generalizations of the Glauberman correspondence and properties thereof. More specifically, we will need to control inertia groups in contexts where these generalized correspondences are defined. When returning to the context of the Glauberman correspondence, we obtain the following result, which, as far as we are aware, has not yet appeared in the literature.

\begin{thmC}
	Suppose $A$ acts on a finite group $G$ by automorphisms. Let $B {\triangleleft}{\triangleleft} A$ be a $p$-subgroup, where $p$ does not divide $|G|$. Let $C = \CB_{G}(B)$ and write $\pi_B$ for the $B$-Glauberman correspondence. Let $C \leq H \leq G$ be an $A$-invariant subgroup, let $\chi \in \Irr_B(G)$ and let $\varphi \in \Irr_B(H)$ be the unique $B$-invariant irreducible character of $H$ such that $\pi_B(\chi) = \pi_B(\varphi)$. Then, $A_\chi = A_\varphi$. Furthermore, for all $p$-subgroups $D \leq A_\chi$ such that $\CB_{G}(D) \leq H$, $\pi_D(\chi) = \pi_D(\varphi)$.
\end{thmC}

In fact, we prove a \emph{relative} version of this result, from which Theorem C follows as a corollary (see \Cref{relThmC}).

\section{Preliminary results}

We begin by collecting some results which will be fundamental to our purposes and that will be used throught the paper. Our notation for characters follows that of \cite{CTFG}. First, there are a few elementary properties of the subnormalizer subgroup that we are going to need. Most of them were originally proved by Casolo, who extensively studied the subnormalizer subset in \cite{CasoloSubnormalizers} and \cite{CasoloSubnormalizersSubgroups}.

\begin{pro}
	\label{SubnormalizerProperties}
	Let $x \in G$ be a $p$-element, where $G$ is a finite group. Then:
	\begin{enumerate}
		\item If $H \leq G$ is any subgroup of $G$ containing $x$ such that $\langle x \rangle {\triangleleft}{\triangleleft} H$, then $H \subseteq S_G(x) \subseteq \Sub_G(x)$. In particular, if $Q \leq G$ is any $p$-subgroup containing $x$, then $\NB_G(Q) \leq \Sub_G(x)$;
		\item For all subgroups $H \leq G$, $\Sub_H(x) \leq \Sub_G(x)$. If $\langle x \rangle {\triangleleft}{\triangleleft} \Sub_G(x)$, then $\Sub_H(x) = \Sub_G(x) \cap H$;
		\item If $N \unlhd G$, then $\Sub_{G/N}(xN) = \Sub_G(x)N/N$;
		\item If $N {\triangleleft}{\triangleleft} G$ and $x \in N$ is a $p$-element, then, $G = N\Sub_G(x)$;
		\item $\Sub_G(x) = S_G(x)$ if and only if $\Sub_G(x) = \NB_G(Q)$, where $Q$ is the intersection of all Sylow $p$-subgroups of $G$ which contain $x$;
		\item If $y \in G$, then $S_G(x)^y = S_G(x^y)$; in particular, $\Sub_G(x)^y = \Sub_G(x^y)$.
	\end{enumerate}
\end{pro}

\begin{proof}
	Recall that $\Sub_G(x) = \langle S_G(x) \rangle$, where $S_G(x) = \{y \in G \mid \langle x \rangle {\triangleleft}{\triangleleft} \langle x, y \rangle\}$.
	\begin{enumerate}
		\item Let $h \in H$. Then, since $\langle x \rangle {\triangleleft}{\triangleleft} H$, we have $\langle x \rangle {\triangleleft}{\triangleleft} H \cap \langle x, h \rangle = \langle x, h \rangle$. Consequently, $h \in S_G(x)$. Since $h$ is arbitrary, $H \subseteq S_G(x)$. For the second part, notice how $\langle x \rangle {\triangleleft}{\triangleleft} Q \unlhd \NB_G(Q)$.
		\item From the definition, it is immediate that $S_H(x) = S_G(x) \cap H$. In particular, $\Sub_H(x) \leq \Sub_G(x)$. On the other hand, if $\langle x \rangle {\triangleleft}{\triangleleft} \Sub_G(x)$, then $\Sub_G(x) = S_G(x)$ and $\Sub_H(x) = S_H(x)$. Then, the equality follows from $S_H(x) = S_G(x) \cap H$.
		\item This follows from \cite[Lemma 2.3]{CasoloSubnormalizers}.
		\item This follows from \cite[Lemma 1.2]{CasoloSubnormalizersSubgroups}.
		\item $\Sub_G(x) = S_G(x)$ if and only if $S_G(x)$ is a subgroup of $G$. By \cite[Proposition 2.1 (b)]{CasoloSubnormalizers}, $S_G(x) = S_G(Q)$. And, by the proof \cite[Proposition 1.4]{CasoloSubnormalizersSubgroups}, if $S_G(Q)$ is a subgroup, then $S_G(Q) = \NB_G(Q)$. The other implication is clear.
		\item We have that $\langle x \rangle {\triangleleft}{\triangleleft} \langle x, a \rangle$ if and only if $\langle x^y \rangle {\triangleleft}{\triangleleft} \langle x^y, a^y \rangle$. Thus, $a \in S_G(x)$ if and only if $a^y \in S_G(x^y)$. The result follows.
	\end{enumerate}
\end{proof}

We remark that second part of item (ii) above does not hold in general without assuming $\langle x \rangle {\triangleleft}{\triangleleft} \Sub_G(x)$. For example, using GAP \cite{GAP} notation, SmallGroup(216, 158) has a $2$-element $x$ of order $2$ which is contained in a subgroup $H$ isomorphic to $S_3$, meaning $\Sub_H(x) = \langle x \rangle$. On the other hand, $H$ is contained in $\Sub_G(x)$, so that $\Sub_G(x) \cap H = H \neq \Sub_H(x)$.

For what follows, we need an elementary property of the characters $\Irr^x(G)$. This is proved in \cite[Lemma 4.3]{PickyPSolvable}, but we separate it here for the reader's convenience.

\begin{lem}
	\label{nonVanishingInvarConst}
	Let $x \in G$, let $\chi \in \Irr^x(G)$ and let $N \unlhd G$. Then, there exists an $\langle x \rangle$-invariant irreducible constituent of $\chi_N$ and, if $\psi$ is the Clifford correspondent of $\chi$ over $\theta$, $\psi(x) \neq 0$.
\end{lem}

\begin{proof}
	Let $\theta \in \Irr(N)$ be any irreducible constituent of $\chi_N$ and let $\psi$ be the Clifford correspondent of $\chi$ over $\theta$. Since $\chi(x) \neq 0$, we have
	\begin{equation*}
		0 \neq \chi(x) = \psi^G(x) = \frac{1}{|G_\theta|} \sum_{g \in G} \psi^\circ (gxg^{-1}),
	\end{equation*}
	by the Clifford correspondence (here, $\psi^\circ(y) = \psi(y)$ if $y \in G_\theta$ and $\psi^\circ(y) = 0$ otherwise). Thus, there exists $g \in G$ such that $gxg^{-1} \in G_\theta$ and $\psi^g(x) \neq 0$. Then, $x \in G_{\theta^g}$, meaning $\theta^g$ is an $\langle x \rangle$-invariant irreducible constituent of $\chi_N$ and the Clifford correspondent of $\chi$ over $\theta^g$ is easily seen to be $\psi^g$.
\end{proof}

We will also make extensive use of another result from \cite{PickyPSolvable}, which relates to coprime action. For the reader's convenience, we include the statement below.

\begin{lem}
	\label{ElementaryCoprimeActionLemma}
	Suppose that $A$ acts by automorphisms on a finite group $G$, let $N \unlhd G$ be $A$-invariant with $(|G:N|, |A|) = 1$ and let $C/N = \CB_{G/N}(A)$.
	\begin{enumerate}
		\item If $\chi \in \Irr(G)$ is $A$-invariant, then $\chi_N$ has an $A$-invariant irreducible constituent, $\theta$, and the set of such constituents is $\{\theta^c \mid c \in C\}$. In particular, if $C = N$, $\theta$ is unique;
		\item If $\theta$ is $A$-invariant, then $\theta^G$ has an $A$-invariant irreducible constituent $\chi$. Furthermore, $\chi$ is unique if $C = N$;
		\item Let $\chi \in \Irr(G)$ and let $\theta \in \Irr(N)$ lie under $\chi$. If $C = G$, then $\chi$ is $A$-invariant if and only if $\theta$ is $A$-invariant.
	\end{enumerate}
\end{lem}

\begin{proof}
	See \cite[Lemma 2.1]{PickyPSolvable}.
\end{proof}

Finally, we will need a well-known result on the existence of a specific character triple isomorphism, known as the \emph{Dade-Turull character correspondence}. In fact, we will need a specific property of this correspondence that has to do with canonical-extensions.

\begin{lem}
	\label{TurullLemma}
	Let $p$ be an odd prime, let $N \unlhd G$ be a subgroup of $p'$-order and suppose $P$ is a $p$-subgroup of $G$ such that $NP \unlhd G$. Let $\theta \in \Irr_P(N)$ and let $\theta^*$ be its $P$-Glauberman correspondent. Then, there exists a character triple isomorphism $*$ from $(G, N, \theta)$ to $(\NB_G(P), \CB_N(P), \theta^*)$ associated to the canonical isomorphism $G/N \cong \NB_G(P)/\CB_N(P)$. This isomorphism $*$ has the following property: if $Q$ is a $p$-subgroup of $G$ such that $\theta$ is $Q$-invariant and $\NB_G(P) = \NB_G(Q)$, then $\theta^*$ is also the $Q$-Glauberman correspondent of $\theta$ and, if $\tau_P$, $\nu_P$ (resp. $\tau_Q, \nu_Q$) denote the canonical extensions of $\theta, \theta^*$ to $NP, \CB_N(P)P$ (resp. $NQ, \CB_N(Q)Q$) respectively, then $\tau_P^* = \nu_P$ and $\tau_Q^* = \nu_Q$.
\end{lem}

\begin{proof}
	By \cite[Lemma 7.3]{Turull} and \cite[Theorem 7.12]{Turull2}, for each subgroup $N \leq H \leq G$, there exists a map $*$ from $\Irr(H \mid \theta)$ to $\Irr(\NB_H(P) \mid \theta^*)$ satisfying an ample number of conditions (conditions (1) through (10) of \cite[Theorem 7.12]{Turull2} and conditions (1) through (5) of \cite[Lemma 7.3]{Turull}). In particular, by definition (\cite[Definition 11.23]{CTFG}, for instance), it, together with the canonical isomorphism $G/N \to \NB_G(P)/\CB_N(P)$, defines a character triple isomorphism.
	
	Since $\NB_G(P) = \NB_G(Q)$, we have $\CB_N(P) = \CB_N(Q)$. By the definition of the Glauberman correspondents, then, it follows that $\theta^*$ is the $Q$-Glauberman correspondent of $\theta$ as well. Now, By \cite[Corollary 6.4]{CTMcKC}, canonical extensions preserve fields of values. Since $N, \CB_N(P)$ are groups of $p'$-order, then, $\tau_P, \tau_Q, \nu_P, \nu_Q$ are all $p$-rational. Since $\CB_N(P)P = \CB_N(P) \times P$ and $p$ is odd, meaning the only $p$-rational character of $P$ is $1_P$, it follows that $\nu_P = \theta^* \times 1_P$ is the only $p$-rational character in $\Irr(\CB_N(P)P \mid \theta^*)$ (an analogous result holds for $Q$). 
	
	Now, by one of the properties of $*$ listed in \cite[Lemma 7.3]{Turull}, $\mathbf{Q}_p(\tau_P) = \mathbf{Q}_p(\tau_P^*)$ and $\mathbf{Q}_p(\tau_Q) = \mathbf{Q}_p(\tau_Q^*)$, where $\mathbf{Q}_p$ denotes the $p$-adic numbers. It is a well-known fact (see, for instance \cite[Exercise II.3.3]{Janusz}) that the only roots of unity in $\mathbf{Q}_p$ for $p \neq 2$ are the $(p-1)$-roots of unity. Since $\tau_P$ is $p$-rational, $\mathbf{Q}_p(\tau_P)$ does not contain any $p$-power root of unity, and so neither does $\mathbf{Q}_p(\tau_P^*)$ (and analogously for $Q$). Then, $\tau_P^*$ is a $p$-rational character in $\Irr(\CB_N(P)P \mid \theta^*)$. By what we saw in the previous paragraph, the only such character is $\nu_P$. The same is true for $Q$.
\end{proof}

\section{On the Glauberman correspondence}

We begin this section by recalling a fundamental correspondence for local representation theory: the \emph{relative Glauberman correspondence}. The following version of it is the one in \cite{PickyPSolvable}. Just as in that paper, if $A$ acts by automorphisms on a finite group $G$, we write $\Irr_A(G)$ for the irreducible characters of $G$ which are fixed by the action of $A$.

\begin{thm}[Relative Glauberman correspondence]
	\label{RelativeGlauberman}
	Suppose that $A$ is a $p$-group acting on a finite group $G$ by automorphisms, that $N \unlhd G$ is $A$-invariant and that $G/N$ has order not divisible by $p$. Let $C/N = \CB_{G/N}(A)$. Then, there is a natural bijection $^*: \Irr_A(G) \to \Irr_A(C)$. In fact, if $\chi \in \Irr_A(G)$, then
	\begin{equation*}
		\chi_C = e\chi^* + p\Delta + \Xi,
	\end{equation*}
	where $e \equiv \pm 1 \pmod{p}$, $\Delta$ and $\Xi$ are characters of $C$ (or zero) and no irreducible constituent of $\Xi$ lies over some $A$-invariant character of $N$. Furthermore, if $\tau \in \Irr_A(C)$, then we can write
	\begin{equation*}
		\tau^G = d \mu + p\Psi + \rho,
	\end{equation*}
	where $\mu \in \Irr_A(G)$ is such that $\mu^* = \tau$, $d \equiv \pm 1 \pmod{p}$, $\Psi$ and $\rho$ are characters of $G$ (or zero) and no irreducible constituent of $\rho$ is $A$-invariant.
\end{thm}

\begin{proof}
	This is \cite[Theorem 2.2]{PickyPSolvable}.
\end{proof}

For our purposes, we will actually need a stronger version of both \Cref{RelativeGlauberman} and the (ordinary) Glauberman correspondence, and begin by proving the latter. Part of the result is stated in \cite[Problem 2.7]{CTMcKC}, but we prove everything for the sake of completeness.

\begin{thm}
	\label{correspondence}
	Let $p$ be a prime and let $G$ be a group of $p'$-order. Suppose $A$ is a $p$-group acting on $G$ by automorphisms and let $C = \CB_G(A)$ and $C \leq H \leq G$ an $A$-invariant subgroup of $G$. Then, there exists a natural bijection $f_A: \Irr_{A}(G) \to \Irr_{A}(H)$. In fact, if $\theta \in \Irr_{A}(G)$, then
	\begin{equation*}
		\theta_H = ef_A(\theta) + p\Delta + \Xi,
	\end{equation*}
	where $e \equiv \pm 1 \pmod{p}$ and $\Delta, \Xi$ are characters of $H$ (or zero) such that no constituent of $\Xi$ is $A$-invariant. Also, if $\tau \in \Irr_{A}(H)$, then 
	\begin{equation*}
		\tau^G = d\mu + p\Lambda + \rho,
	\end{equation*}
	where $f_A(\mu) = \tau$, $d \equiv \pm 1 \pmod{p}$ and $\Lambda, \rho$ are characters of $G$ (or zero) such that no constituent of $\rho$ is $A$-invariant. Moreover, writing $\varphi = f_A(\theta)$, then, if $*$ denotes the $A$-Glauberman correspondence, $\theta^* = \varphi^*$ and $[\theta_C, \theta^*] \equiv [\theta_H, \varphi][\varphi_C, \theta^*] \pmod{p}$.
\end{thm}

\begin{proof}
	First, there is, indeed a bijection between $\Irr_{A}(G)$ and $\Irr_{A}(H)$ by the Glauberman correspondence applied twice, since $\CB_H(A) = \CB_G(A)$, as $\CB_G(A) \subseteq H$. That is to say, if $\theta \in \Irr_{A}(G)$, then there exists a unique $\varphi \in \Irr_{A}(H)$ such that $\theta^* = \varphi^*$. We will show that $\varphi$ is under the conditions outlined above.
	
	Write $\theta_H = \sum_{\eta \in \Irr(H)} a_\eta \eta$. Since $H$ is $A$-invariant, $A$ acts on $\Irr(H)$, so that we can write $\Irr(H) = \mathcal{O}_1 \sqcup \cdots \sqcup \mathcal{O}_k$ partitioned by $A$-orbits. Also, $\theta$ is $A$-invariant, meaning $[\theta_H, \eta] = [\theta_H, \eta^x]$ for all $x \in A$. Thus, the $a_\eta$ are, in fact, constant on each orbit. Let $\eta_i$ be a representative of $\mathcal{O}_i$ for each $i$ and let $a_i = [\theta_H, \eta_i]$. Then, $\theta_H = \sum_{i = 1}^k a_i \sum_{\eta \in \mathcal{O}_i} \eta$. Notice that, given $\eta \in \Irr(H)$, $(\eta^x)_{C} = \eta_{C}$ for all $x \in A$. Consequently, for all $\eta \in \mathcal{O}_i$, $\eta_{C} = {\eta_i}_{C}$.
	
	Write $\theta_{C} = m\theta^* + p\Delta'$ as in the $A$-Glauberman correspondence. Then, we have
	\begin{equation}
		\label{eq1}
		m\theta^* + p\Delta' = \theta_{C} = \sum_{i = 1}^k a_i \sum_{\eta \in \mathcal{O}_i} \eta_{C} = \sum_{i = 1}^k |\mathcal{O}_i| a_i {\eta_i}_{C}.
	\end{equation}
	From the fact that $m \not \equiv 0 \pmod{p}$, it follows that there is at least one of the $\mathcal{O}_i$ of size $1$ (since $A$ is a $p$-group, they are all powers of $p$). Without loss of generality, we may assume that orbits $1$ through $l$ are the ones of size $1$ (where $1 \leq l \leq k$). Then, \Cref{eq1} becomes
	\begin{equation*}
		m\theta^* + p\Delta' = \sum_{i = 1}^l a_i {\eta_i}_{C} + p \tilde{\Delta},
	\end{equation*}
	for some (possibly zero) character $\tilde{\Delta}$. For $1 \leq i \leq l$, we may write ${\eta_i}_{C} = m_i\eta_i^* + p\Delta_i$ as in the $A$-Glauberman correspondence. Then, we obtain
	\begin{equation}
		\label{eq2}
		m\theta^* + p\Delta' = \sum_{i = 1}^l a_i m_i\eta_i^*+ p \Delta'',
	\end{equation}
	where, again, $\Delta''$ is a character or $0$. Since the $\eta_i$ are pairwise distinct, so are the $\eta_i^*$. Thus, by \Cref{eq2}, there exists exactly one $i_0$ between $1$ and $l$ such that $\eta_{i_0}^* = \theta^*$ and we must also have $a_{i_0} m_{i_0} \equiv m \equiv \pm 1 \pmod{p}$. In particular, since $m_{i_0} \equiv \pm 1 \pmod{p}$, we have that $a_{i_0} \equiv \pm 1 \pmod{p}$. If $i \neq i_0$ is between $1$ and $l$, then $a_i \equiv 0 \pmod{p}$, again by \Cref{eq2}. And, if $l < i \leq k$, then $\eta_i$ is not $A$-invariant. This gives us the first expression in the statement of the theorem.
	
	For the other direction, let $\tau \in \Irr_{A}(H)$ and let $\mu \in \Irr_{A}(G)$ be such that $f_A(\mu) = \tau$. Then, we know that
	\begin{equation*}
		\mu_H = d\tau + p\Delta + \Xi
	\end{equation*}
	for some $d \equiv \pm 1 \pmod{p}$, and some characters $\Delta, \Xi$, the latter of which does not have $\tau$ as a constituent. Then,
	\begin{equation*}
		\tau^G = d\mu + p\Lambda + \rho
	\end{equation*}
	for some characters $\Lambda, \rho$ (or zero), by Frobenius reciprocity. We may assume that, for all constituents $\xi$ of $\rho$, $[\tau^G, \xi] \not \equiv 0 \pmod{p}$. Suppose there exists one such constituent $\xi$ which is $A$-invariant. Then,
	\begin{equation*}
		\xi_H = mf_A(\xi) + p\Phi + \Psi,
	\end{equation*}
	using the direction we already proved. Since $\xi \neq \mu$ and $f_A$ is bijective, $f_A(\xi) \neq \tau$. Also, $\tau$ is $A$-invariant. Then, by Frobenius reciprocity, since $0 \neq [\tau^G, \xi] = [\tau, \xi_H]$, we must have $[\tau, \xi_H] \equiv 0 \pmod{p}$, a contradiction.
\end{proof}

Now, we do the same but for the relative Glauberman correspondence. Of course, \Cref{relativeCorrespondence} generalizes \Cref{correspondence}, but we use \Cref{correspondence} in its proof.

\begin{thm}
	\label{relativeCorrespondence}
	Suppose that $A$ is a $p$-group acting on a finite group $G$ by automorphisms, that $N \unlhd G$ is $A$-invariant and that $G/N$ has order not divisible by $p$. Let $C/N = \CB_{G/N}(A)$ and let $C \leq H \leq G$ be an $A$-invariant subgroup of $G$. Then, there is a natural bijection $f_A: \Irr_A(G) \to \Irr_A(H)$. In fact, if $\chi \in \Irr_A(G)$, then
	\begin{equation*}
		\chi_H = ef_A(\chi) + p\Delta + \Xi,
	\end{equation*}
	where $e \equiv \pm 1 \pmod{p}$, $\Delta$ and $\Xi$ are characters of $H$ (or zero) and no irreducible constituent of $\Xi$ is $A$-invariant. Moreover, writing $\psi = f_A(\chi)$, if $*$ denotes the $A$-relative Glauberman correspondence, then $\chi^* = \psi^*$ and $[\chi_C, \chi^*] \equiv [\chi_H, \psi][\psi_C, \chi^*] \pmod{p}$. Furthermore, if $\tau \in \Irr_A(H)$, then we can write
	\begin{equation*}
		\tau^G = d \mu + p\Psi + \rho,
	\end{equation*}
	where $\mu \in \Irr_H(G)$ is such that $f_A(\mu) = \tau$, $d \equiv \pm 1 \pmod{p}$, $\Psi$ and $\rho$ are characters of $G$ (or zero) and no irreducible constituent of $\rho$ is $A$-invariant.
\end{thm}

\begin{proof}
	We begin by proving the direction going down, just as was done in \Cref{correspondence}. For this, we follow the general outlines of the proof of \cite[Theorem E]{RelativeGlauberman}. Just as in that theorem, $C$ induces partitions
	\begin{equation*}
		\Irr_A(G) = \bigsqcup_{\theta \in \Lambda} \Irr_A(G \mid \theta) \text{ and } \Irr_A(H) = \bigsqcup_{\theta \in \Lambda} \Irr_A(H \mid \theta),
	\end{equation*}
	where $\Lambda$ is a complete set of representatives of the action of $C$ on $\Irr_A(N)$. We then strive for a bijection relative to some $\theta \in \Irr_A(N)$, which we construct by induction on $|G:N|$.
	
	Let $\chi \in \Irr_A(G \mid \theta)$ and let $T = G_\theta$. Following the argument of \cite[Theorem E]{RelativeGlauberman}, we may write $\chi_H = (\psi_{H \cap T})^H + \delta$, where $\psi \in \Irr_A(T \mid \theta)$ is the Clifford correspondent of $\chi$ and no irreducible constituent of $\delta$ lies over any $A$-invariant character of $N$. In particular, by \Cref{ElementaryCoprimeActionLemma}, no irreducible constituent of $\delta$ can be $A$-invariant.
	
	If $T < G$, by our induction hypothesis, we may write $\psi_{H \cap T} = e \tilde{f}_A(\psi) + p\Delta + \Xi$, where $\tilde{f}_A$ satisfies the conditions in the statement and $e \equiv \pm 1 \pmod{p}$. So we get
	\begin{equation*}
		\chi_H = e \tilde{f}_A(\psi)^H + p\Delta^H + \Xi^H + \delta.
	\end{equation*}
	Notice that all irreducible constituents of $\psi_{H\cap T}$ lie over $\theta$, so that $\tilde{f}_A(\psi)^H$ is irreducible by the Clifford correspondence, as are the inductions of the constituents of $\Delta, \Xi$. Also, if $\xi$ is an irreducible constituent of $\Xi$ and $a \in A$, $(\xi^H)^a = (\xi^a)^H$ and $\xi^a \in \Irr(H \cap T \mid \theta)$. Thus, if $\xi^H = (\xi^H)^a$, by the bijectivity of the Clifford correspondence, it follows that $\xi = \xi^a$. As $\xi$ is not $A$-invariant, it follows that $\xi^H$ is also not $A$-invariant.
	
	Then, $f_A(\chi) := \tilde{f}_A(\psi)^H$ is the unique irreducible constituent of $\chi_H$ whose multiplicity is not divisible by $p$ and which is $A$-invariant. Also, $[\chi_H, f_A(\chi)] \equiv [\psi_{H \cap T}, \tilde{f}_A(\psi)] \pmod{p}$.
	
	Write $\chi_T = \psi + R$, where no irreducible constituent of $R$ lies over $\theta$, by the Clifford correspondence. Then, $\chi_C = \psi_C + R_C$. By \Cref{RelativeGlauberman}, $\psi_C = m\psi^* + pS$, with $m \equiv \pm 1 \pmod{p}$ -- notice that every irreducible constituent of $\psi_C$ lies over $\theta$, which is $A$-invariant; hence these are the only summands of the correspondence in \Cref{RelativeGlauberman} that appear. 
	
	Then, $\chi_C = m\psi^* + pS + R_C$. Since $\psi^*$ lies over $\theta$, it may not be an irreducible constituent of $R_C$. Then, $[\chi_C, \psi^*] = [\psi_C, \psi^*] \equiv m \equiv \pm 1 \pmod{p}$. Also, $\psi^*$ is $A$-invariant. By the uniqueness of the $A$-relative Glauberman correspondence, $\chi^* = \psi^*$. 
	
	Following the same argument, since we defined $f_A(\chi)$ as the Clifford correspondent of $\tilde{f}_A(\psi)$, it follows that $f_A(\chi)^* = \tilde{f}_A(\psi)^*$. By our hypothesis on $\tilde{f}_A$, then, we have $f_A(\chi)^* = \tilde{f}_A(\psi)^* = \psi^* = \chi^*$. Finally, we have $[\chi_C, \chi^*] = [\psi_C, \psi^*]$ and $[f_A(\chi)_C, \chi^*] = [\tilde{f}_A(\psi)_C, \psi^*]$ -- this latter equality follows once again by applying the same argument of the previous two paragraphs to $f_A(\chi)$. We proved above that $[\psi_{H \cap T}, \tilde{f}_A(\psi)] \equiv [\chi_H, f_A(\chi)] \pmod{p}$. Putting everything together and using our hypothesis yields
	\begin{equation*}
		[\chi_C, \chi^*] = [\psi_C, \psi^*] \equiv [\psi_{H \cap T}, \tilde{f}_A(\psi)][\tilde{f}_A(\psi)_C, \psi^*] \equiv [\chi_H, f_A(\chi)][f_A(\chi)_C, \chi^*] \pmod{p},
	\end{equation*}
	which is what we wanted to show. Thus, we may assume that $T = G$.
	
	Let $\Gamma = G\rtimes A$. Then, $\theta$ is $\Gamma$-invariant. Thus, we may consider a character triple $(\Gamma^\dagger, N^\dagger, \theta^\dagger)$ isomorphic to $(\Gamma, N, \theta)$, such that $N^\dagger \subseteq \ZB(\Gamma^\dagger)$. We will write (simultaneously, with some abuse of notation) $\dagger$ for the isomorphism $\Gamma/N \to \Gamma^\dagger/N^\dagger$ and bijections of characters $\Irr(J \mid \theta) \to \Irr(J^\dagger \mid \theta^\dagger)$, where $N \leq J \leq \Gamma$ and $J/N \cong J^\dagger/L^\dagger$.
	
	Just as in \cite[Theorem E]{RelativeGlauberman}, $(NA)^\dagger$ has a unique Sylow $p$-subgroup $A^\dagger$ which acts on $G^\dagger/N^\dagger$ as $A$ does on $G/N$. So $\dagger$ maps $\Irr_A(J \mid \theta)$ onto $\Irr_{A^\dagger}(J^\dagger \mid \theta^\dagger)$. Assume that there is a bijection $f_{A^\dagger}:\Irr_{A^\dagger}(G^\dagger \mid \theta^\dagger) \to \Irr_{A^\dagger}(H^\dagger \mid \theta^\dagger)$ satisfying the conditions we want. Then, we may define $f_A: \Irr_A(G \mid \theta) \to \Irr_A(H \mid \theta)$ as the unique bijection such that $f_{A^\dagger}(\chi^\dagger) = f_A(\chi)^\dagger$. We claim this is the map we want.
	
	Let $\chi \in \Irr_A(G \mid \theta)$. Then, by the properties of $f_{A^\dagger}$ and of character triple isomorphisms, we have
	\begin{equation*}
		(\chi_H)^\dagger = ef_{A^\dagger}(\chi^\dagger) + p\Delta^\dagger + \Xi^\dagger =  (ef_{A}(\chi) + p\Delta+ \Xi)^\dagger,
	\end{equation*}
	where $e \equiv \pm 1 \pmod{p}$, and no irreducible constituent of $\Xi^\dagger$ is $A^\dagger$-invariant, which is equivalent to no constituent of $\Xi$ being $A$-invariant. Furthermore, we may write $\chi_C = m_1 \chi^* + p\Delta_1$ and $\chi^\dagger_{C^\dagger} = m_2 (\chi^\dagger)^* + p\Delta_2$, by \Cref{RelativeGlauberman}. From the properties of character triple isomorphisms, it follows from this that $(\chi^\dagger)^* = (\chi^*)^\dagger$. Then, we get
	\begin{equation*}
		(\chi^*)^\dagger = (\chi^\dagger)^* = (f_{A^\dagger}(\chi^\dagger))^* = (f_A(\chi)^\dagger)^* = (f_A(\chi)^*)^\dagger,
	\end{equation*}
	meaning $\chi^* = f_A(\chi)^*$. Finally, 
	\begin{align*}
		[\chi_C, \chi^*] &= [\chi^\dagger_{C^\dagger}, (\chi^\dagger)^*] \equiv [\chi^\dagger_{H^\dagger}, f_{A^\dagger}(\chi^\dagger)][f_{A^\dagger}(\chi^\dagger), (\chi^\dagger)^*] \\
		&= [\chi_H, f_A(\chi)][f_A(\chi), \chi^*] \pmod{p},
	\end{align*}
	as desired.
	
	So, to finish the proof, we may assume that $N \subseteq \ZB(\Gamma)$. In this case, since $G/N$ is a $p'$-group, $G$ has a central Sylow $p$-subgroup $N_p$, meaning we may write $G = X \times N_p$ for a normal $p$-complement $X$ and consequently, $H = (X \cap H) \times N_p$, $C = (X \cap C) \times N_p$. Also, notice that $C = \CB_G(A)$ and $X \cap C = \CB_X(A)$. Furthermore, we may write $\Irr_A(G \mid \theta) = \{\mu \times \theta_p \mid \mu \in \Irr_A(X \mid \theta_{p'})\}$ (similar equalities hold for $H, C$). 
	
	By \Cref{correspondence}, using that $\theta_{p'}$ is $\Gamma$-invariant, we have a bijection $\tilde{f}_A: \Irr_A(X \mid \theta_{p'}) \to \Irr_A(X \cap H \mid \theta_{p'})$ satisfying the conditions that we want. Then, we may define $f_A: \Irr_A(G \mid \theta) \to \Irr_A(H \mid \theta)$ by $f_A(\mu \times \theta_p) = \tilde{f}_A(\mu) \times \theta_p$.
	
	Let $\chi = \mu \times \theta_p \in \Irr_A(G \mid \theta)$. Then,
	\begin{equation*}
		\chi_H = (\mu_{X \cap H}) \times \theta_p = e f_A(\chi) + p\Delta \times \theta_p + \Xi \times \theta_p
	\end{equation*}
	with $e \equiv \pm 1 \pmod{p}$ and no irreducible constituent of $\Xi$ being $A$-invariant. Since $\theta_p$ is $A$-invariant, $\xi \times \theta_p$ is $A$-invariant if and only if $\xi$ is. Hence, $f_A(\chi)$ is the unique $A$-invariant irreducible constituent of $\chi_H$ whose multiplicity is not divisible by $p$. 
	
	Notice how, since $(\mu \times \theta_p)_C = (\mu_{\CB_X(A)}) \times \theta_p$, then $\mu^* \times \theta_p$ is the $A$-relative Glauberman correspondent of $\mu \times \theta_p$ (we are using the same notation for the relative and non-relative Glauberman correspondences). Then, $\chi^* = \mu^* \times \theta_p = \tilde{f}_A(\mu)^* \times \theta_p = f_A(\chi)^*$. Finally, 
	\begin{align*}
		[\chi_C, \chi^*] &= [\mu_{\CB_X(A)}, \mu^*] \equiv [\mu_{X \cap H} \tilde{f}_A(\mu)][\tilde{f}_A(\mu)_{\CB_X(A)}, \mu^*] \\
		&\equiv [\chi_H, f_A(\chi)][f_A(\chi)_C, \chi^*] \pmod{p},
	\end{align*}
	as desired. This finishes the proof of one direction of the correspondence. The other direction follows analogously to the second part of \Cref{correspondence}.
\end{proof}

For the remainder of the paper, we will fix the notation $f_A$ to mean the correspondence in \Cref{relativeCorrespondence}. Since this generalizes \Cref{correspondence}, there shall be no ambiguity. When $A = \langle x \rangle$ is cyclic, we will just write $f_A = f_x$.

We will need to control how certain inertia groups relate to one another under the correspondence $f_x$. We thus establish the following general result, which is the relative version of Theorem C from the Introduction. We will make extensive use of it in the remaining sections, and it will be the key for the construction of the bijections we are after.

\begin{thm}
	\label{relThmC}
	Suppose $A$ acts on a finite group $G$ by automorphisms and let $N \unlhd G$ be an $A$-invariant subgroup. Let $\hat{A} = A \ltimes N$ and let $B {\triangleleft}{\triangleleft} \hat{A}$ be a $p$-subgroup, where $p$ does not divide $|G:N|$. Let $C/N = \CB_{G/N}(B)$. Let $C \leq H \leq G$ be an $A$-invariant subgroup, let $\chi \in \Irr_B(G)$ and let $f_B(\chi) = \varphi \in \Irr_B(H)$ its correspondent via \Cref{correspondence}. Then, $A_\chi = A_\varphi$. Furthermore, for all $p$-subgroups $D \leq A_\chi$ such that $\CB_{G/N}(D) \leq H/N$, $f_D(\chi) = \varphi$.
\end{thm}

\begin{proof}
	Notice that $\varphi = f_B(\chi)$, and that $\hat{A}_\chi = A_\chi N$. Hence, it suffices to show $\hat{A}_\chi = \hat{A}_\varphi$. Since $\chi$ is $B$-invariant and $B{\triangleleft}{\triangleleft}\hat{A}$, we have $B{\triangleleft}{\triangleleft}\hat{A}_\chi$. Let $B \lhd A_1 \lhd \cdots \lhd A_k = \hat{A}_\chi$ be a subnormal series for $B$. We show, by induction on the length on this subnormal chain, that $\varphi$ is $\hat{A}_\chi$-invariant.
	
	Let $a \in A_1$ and let $b \in B$ be arbitrary. Then, since $B \lhd A_1$, there exists $b'\in B$ such that $aba^{-1} = b'$. Consequently, $\varphi^{ab} = \varphi^{b'a} = \varphi^a$, since $\varphi$ is $B$-invariant. Also, $[\chi_H, \varphi] = [\chi_H, \varphi^a] \not \equiv 0 \pmod{p}$, meaning $\varphi^a$ is a $B$-invariant irreducible constituent of $\chi_H$ with multiplicity not divisible by $p$. By \Cref{correspondence}, we have $\varphi = \varphi^a$. Since $a$ is arbitrary, $\varphi$ is $A_1$-invariant.
	
	Now suppose $\varphi$ is $A_i$-invariant and let $a \in A_{i + 1}$, $x \in A_i$. Again, $A_i \lhd A_{i+1}$, meaning there exists $x' \in A_i$ such that $axa^{-1} = x'$. Then, $\varphi^{ax} = \varphi^{x'a} = \varphi^a$ by our inductive hypothesis. The same reasoning as before allows us to conclude $\varphi = \varphi^a$. Hence, $\hat{A}_\chi \subseteq \hat{A}_\varphi$. The other inclusion follows from the other direction of \Cref{correspondence}.
	
	Let, now, $D \leq A_\chi$ be a $p$-subgroup such that $\CB_{G/N}(D) \leq H/N$. Then, $H$ is $D$-invariant by hypothesis and the correspondence of \Cref{correspondence} is defined. Moreover, the correspondent of $\chi$ via $f_D$ is the unique irreducible constituent of $\chi_H$ which is $D$-invariant and has multiplicity not divisible by $p$. Since $A_\chi = A_\varphi$, $\varphi$ is $D$-invariant. Thus, by uniqueness, $\varphi = f_D(\chi)$, as we wanted to show.
\end{proof}

\begin{thmC}
	Suppose $A$ acts on a finite group $G$ by automorphisms. Let $B {\triangleleft}{\triangleleft} A$ be a $p$-subgroup, where $p$ does not divide $|G|$. Let $C = \CB_{G}(B)$. Write $\pi_B$ for the $B$-Glauberman correspondence. Let $C \leq H \leq G$ be an $A$-invariant subgroup, let $\chi \in \Irr_B(G)$ and let $f_B(\chi) = \varphi \in \Irr_B(H)$ be its correspondent via \Cref{correspondence}. Then, $A_\chi = A_\varphi$. Furthermore, for all $p$-subgroups $D \leq A_\chi$ such that $\CB_{G}(D) \leq H$, $f_D(\chi) = \varphi$.
\end{thmC}

\begin{proof}
	Follows from \Cref{relThmC} using $N = 1$.
\end{proof}

We highlight a particular case of Theorem C, which is how it will most often be used throughout the paper. It can be thought of as a version of Theorem C where all of its parts lie inside some bigger group $G$.

\begin{cor}
	\label{inertiaGroupLemma}
	Let $L \leq K \leq G$ be such that $L, K \unlhd G$ and $K/L$ has $p'$-order. Let $x \in G$ be a $p$-element and let $H \leq G$ be such that $H \cap K \supseteq C$, where $C/L = \CB_{K/L}(x)$. Suppose that $L\langle x \rangle {\triangleleft}{\triangleleft} AL$, where $A \leq \NB_G(H \cap K)$. Let $\theta \in \Irr_{\langle x \rangle}(K)$ and let $\varphi \in \Irr_{\langle x \rangle}(H \cap K)$ be its correspondent via \Cref{relativeCorrespondence}. Then, $A_\theta = A_\varphi$. Furthermore, if $y$ is any $p$-element in $A_\theta$ such that $\CB_{K/L}(y) \subseteq (H \cap K)/L$, then $f_y(\theta) = \varphi$. In particular, this holds for the $A$-conjugates of $x$ in $A_\theta$.
\end{cor}

\begin{proof}
	Follows from Theorem C by taking $B = L\langle x \rangle$ and noticing that a character of $K$ (resp. $H \cap K$) is $\langle x \rangle$-invariant if and only if it is $L\langle x \rangle$-invariant.
\end{proof}

\section{Normal $p$-complement}
\label{Section 3}

In this section, we tackle the strong subnormalizer conjecture in the case where $G$ has a normal $p$-complement. This will lay the groundwork for what will follow in \Cref{Section 4}. Before doing so, however, we take the opportunity to describe the structure of the subnormalizer, which seems to be substantially more restricted than in the general setting.

\begin{lem}
	\label{SubnormalizerNormalComplement}
	Suppose $G$ has a normal $p$-complement $N$, and let $x \in G$ be a $p$-element. Let $P \in \Syl_p(G)$ contain $x$. Then, $\Sub_G(x) = \langle P, \CB_N(x) \rangle$.
\end{lem}
\begin{proof}
	Let $G$ have a normal $p$-complement $N$ and write $G = NP$, where $P \in \Syl_p(G)$ contains the $p$-element $x$. Then, by \cite[Proposition 2.6]{GunterSubnormalizer},
	\begin{equation*}
		\Sub_G(x) = \langle \NB_G(Q) \mid x \in Q \in \Syl_p(G) \rangle = \langle Q \CB_N(Q) \mid x \in Q \in \Syl_p(G) \rangle.
	\end{equation*}
	We claim that, for all $Q \in \Syl_p(G)$ such that $x \in Q$, then $Q\CB_N(Q) \leq \langle P, \CB_N(Q) \rangle$. Indeed, let $Q \in \Syl_p(G)$ contain $x$ and let $n \in N$ be such that $Q^n = P$. Then, both $x^n$ and $x^{-1}$ belong to $P$. As such, we have $[x, n] = x^{-1}x^n \in N \cap P = 1$, so that $n \in \CB_N(x)$ and $Q \leq \langle P, \CB_N(x) \rangle$. Since $\CB_N(Q) \leq \CB_N(x)$, as $x \in Q$, $Q\CB_N(Q) \leq \langle P, \CB_N(x) \rangle$, proving the claim. Thus,
	\begin{equation*}
		\Sub_G(x) = \langle P, \CB_N(x) \rangle,
	\end{equation*}
	as desired.
\end{proof}

Using \Cref{SubnormalizerNormalComplement}, notice that $x \in \ZB(\Sub_G(x))$ if and only if $x \in \ZB(P)$ for a Sylow $p$-subgroup $P$ containing $x$. Even further, $\langle x \rangle \unlhd \Sub_G(x)$ if and only if $\langle x \rangle \unlhd P$. These conditions are far from holding in general; it is even the case that $x$ can be central in $\langle P \in \Syl_p(G) \mid x \in P \rangle$ and still fail to be central in $\Sub_G(x)$. For example, there is a $2$-element $x$ of SmallGroup(360, 121) ($= A_5 \times S_3$) which is central in the group generated by the Sylows containing it, but $\langle x \rangle$ is not even normal in $\Sub_G(x)$.

We should also remark that it is not true that $\langle x \rangle$ is always subnormal in $\Sub_G(x)$ for $G$ with a normal $p$-complement. For example, SmallGroup(72, 40) has a $2$-element $x$ whose subnormalizer is the group itself, but $\langle x \rangle$ is not subnormal in it.

Now, we use the preceding result to determine the normal $p$-complement of $\Sub_G(x)$. In order to do so, we need a bit of notation. If $G$ is a finite group and $H \leq G$, recall that the \emph{normal closure} of $H$ in $G$, which we denote $H^G$, is the unique smallest normal subgroup of $G$ containing $H$. It is easy to see that $H^G = \langle H^x \mid x \in G \rangle$.

\begin{cor}
	\label{complementSubnormalizerNormalComplement}
	Suppose $G$ has a normal $p$-complement $N$ and let $x \in G$ be a $p$-element. Let $P \in \Syl_p(G)$ contain $x$ and write $S = \Sub_G(x)$. Then, $N \cap S = \CB_N(x)^S$, which is a normal $p$-complement of $S$, and $S = \Sub_G(x) = \CB_N(x)^SP$.
\end{cor}

\begin{proof}
	By \Cref{SubnormalizerNormalComplement}, $\CB_N(x)^S = \langle \CB_N(x)^y \mid y \in P \rangle$. Thus, $\CB_N(x)^S = \langle \CB_N(x^y) \mid y \in P \rangle \subseteq N$ is a $p'$-group generated by the centralizers in $N$ of the $P$-conjugates of $x$. Also, by \Cref{SubnormalizerNormalComplement}, the homomorphism $P \to S/\CB_N(x)^S$ sending $y$ to $y\CB_N(x)^S$ is surjective, implying that $\CB_N(x)^S$ is a normal $p$-complement for $S = \Sub_G(x)$. Then, $\CB_N(x)^S = \Sub_G(x) \cap N$.
\end{proof}

One fact which becomes more readily apparent from \Cref{complementSubnormalizerNormalComplement} is that the normal $p$-complement of $\Sub_G(x)$ lies above $\CB_N(y)$ for all $P$-conjugates $y$ of $x$. This allows us to deal with multiple Glauberman correspondences at once through the correspondence in \Cref{correspondence}, which will be a crucial observation in what is to come.

We now prove the key result for the proof of the strong subnormalizer conjecture for groups with a normal $p$-complement. In fact, we are able to prove a bit more that we need, showing that the bijection preserves values (up to a sign) of more elements than just $x$.

\begin{thm}
	\label{normalPComplement}
	Let $p$ be an odd prime, let $G$ have a normal $p$-complement $N$, let $x$ be a $p$-element of $G$ and let $S = \Sub_G(x)$. Let $\theta \in \Irr_{\langle x \rangle}(N)$ and let $P \in \Syl_p(G)$ contain $x$. Then, there exists a bijection $$F: \Irr(G \mid \theta) \to \Irr(\Sub_G(x) \mid \varphi)$$ where $\varphi = f_x(\theta) \in \Irr_{\langle x \rangle}(N \cap S)$, such that $F(\chi)(1)_p = \chi(1)_p$ and, for all $y \in P_\theta$ such that $\CB_N(y) \subseteq N \cap S$, $F(\chi)(y) = \epsilon \chi(y)$ where $\epsilon$ is the unique sign such that $[\theta_{N \cap S}, \varphi] \equiv \epsilon \pmod{p}$. In particular, this holds for all $y \in P_{\theta}$ which are $P$-conjugate to $x$.
\end{thm}

\begin{proof}
	Let $T = G_\theta$, $S = \Sub_G(x)$, $C = \CB_N(x)^S$ and $\varphi = f_x(\theta) \in \Irr_{\langle x \rangle}(C)$. Write $Q = P_\theta$. Since $N \subseteq T$, we have $T = (NP)_\theta = NP_\theta = NQ$. Also, $G = NP$ and $S = CP$ by \Cref{complementSubnormalizerNormalComplement}, and $Q = P_\theta = P_\varphi$, by \Cref{inertiaGroupLemma} applied with $L = 1$, $K = N$, $H = \Sub_G(x)$ and $A = P$. Furthermore, $S \cap T = (CP) \cap T = CP_\theta = CP_\varphi = S_\varphi$, by \Cref{inertiaGroupLemma}.
	
	By \cite[Corollary 6.28]{CTFG}, $\theta$ and $\varphi$ have canonical extensions $\widehat{\theta}, \widehat{\varphi}$ to $T$ and $S_\varphi$, respectively. Then, by the Gallagher correspondence, identifying the characters of $P_\theta$ with those of $T/N$ and $S_\varphi/C$, there exist bijections $\Irr(P_\theta) \to \Irr(T\mid \theta)$ and $\Irr(P_\theta) \to \Irr(S_\varphi\mid \varphi)$ given by $\mu \mapsto \widehat{\theta}\mu$ and $\mu \mapsto \widehat{\varphi}\mu$, respectively, and this gives us a bijection $\tilde{F}: \Irr(T\mid \theta) \to \Irr(S_\varphi \mid \varphi)$ defined as $\tilde{F}(\widehat{\theta}\mu) = \widehat{\varphi}\mu$. 
	
	Notice that $\tilde{F}(\widehat{\theta}\mu)(1)_p = \mu(1) = (\widehat{\theta}\mu)(1)_p$. We now need to determine the values of $\widehat{\theta}(y)$ and $\widehat{\varphi}(y)$ for all $y \in Q$ such that $\CB_N(y) \leq C = N \cap S$. For those $y$, by \Cref{inertiaGroupLemma}, $\varphi = f_y(\theta)$.
	
	Let $\pi_y$ denote the $\langle y \rangle$-Glauberman correspondence. With a slight abuse of notation, we will use this same symbol both for the map $\Irr_{\langle y \rangle}(N) \to \Irr(\CB_N(y))$ as well as for the map $\Irr_{\langle y \rangle}(C) \to \Irr(\CB_N(y))$. Let $V = N\langle y \rangle$ and $U = C\langle y \rangle$. Then, $\widehat{\theta}_V$ (resp., $\widehat{\varphi}_U$) is the canonical extension of $\theta$ to $V$ (resp., $\varphi$ to $U$). This means $\widehat{\theta}(y) = \epsilon \pi_y(\theta)(1)$, where $\epsilon \in \{1, -1\}$ is, since $p \neq 2$, the unique sign such that $[\theta_{\CB_N(y)}, \pi_y(\theta)] \equiv \epsilon \pmod{p}$, by \cite[Theorem 13.6]{CTFG}. Analogously, $\widehat{\varphi}(y) = \delta \pi_y(\theta)(1)$ for some sign $\delta \in \{1, -1\}$ such that $[\varphi_{\CB_N(y)}, \pi_y(\theta)] \equiv \delta \pmod{p}$. Then, we get
	\begin{equation*}
		\tilde{F}(\widehat{\theta}\mu)(y) = \widehat{\varphi}(y)\mu(y) = \delta\pi_y(\theta)(1) \mu(y) = \delta\epsilon(\widehat{\theta}\mu)(y).
	\end{equation*}
	Notice how $\epsilon \equiv [\theta_{\CB_N(y)}, \pi_y(\theta)] \equiv [\theta_C, \varphi]\delta \pmod{p}$, by \Cref{correspondence}. This means that $\epsilon \delta$ is the unique sign such that $[\theta_C, \varphi] \equiv \epsilon\delta \pmod{p}$ and is, therefore, independent of $y$. We will write $\tilde{\epsilon} = \epsilon \delta$.
	
	By the Clifford correspondence, we also have bijections $\Irr(T\mid\theta) \to \Irr(G\mid\theta)$ and $\Irr(S_\varphi \mid \varphi) \to \Irr(S\mid \varphi)$ given by induction of characters. Combining these with $\tilde{F}$, we construct a bijection $F: \Irr(G \mid \theta) \to \Irr(S\mid \varphi)$ given by the property $F(\psi^G) = (\tilde{F}(\psi))^S$, which we claim to be the map we want.
	
	First, $G = TS$, since $N \subseteq T$ and $P \subseteq S$. Thus we have that
	\begin{equation*}
		F(\psi^G)(1)_p = |S:S\cap T|_p \tilde{F}(\psi)(1)_p = |G:T|_p \left(\frac{\psi^G(1)_p}{|G:T|_p}\right) = \psi^G(1)_p,
	\end{equation*}
	as we wanted to show.
	
	Next, let $y \in Q$ be such that $\CB_N(y) \leq C$. By the induction formula, we may write
	\begin{equation*}
		\psi^G(y) = \frac{1}{|T|} \sum_{\substack {g \in G \\ gyg^{-1} \in T}} \psi(gyg^{-1}).
	\end{equation*}
	Since $G = NP$ and $N \cap P = 1$, each $g \in G$ can be written uniquely as $g = nh$, $n \in N$ and $h \in P$. Then, we get
	\begin{equation*}
		\psi^G(y) = \frac{1}{|T|} \sum_{n \in N} \sum_{h \in P} \psi^\circ(nhyh^{-1}n^{-1}) = \frac{|N|}{|T|} \sum_{\substack{h \in P \\ hyh^{-1} \in T}} \psi(hyh^{-1}),
	\end{equation*}
	where the second equality follows from the fact that $N \leq T$, meaning $nhyh^{-1}n^{-1} \in T$ if, and only if, $hyh^{-1} \in T$.
	
	Since $y \in Q = P_\theta$ and $h \in P$, then $hyh^{-1} \in P$. Thus, $hyh^{-1}$ is in $T$ if, and only if, it is in $T \cap P = Q$. Also, $\CB_N(hyh^{-1}) = \CB_N(y)^{h^{-1}} \leq C^{h^{-1}} = C$, since $C$ is normal in $S$ and $h \in P \leq S$. By the properties of $\tilde{F}$, we have $\psi(hyh^{-1}) = \tilde{\epsilon} \tilde{F}(\psi)(hyh^{-1})$ for a sign $\tilde{\epsilon}$ which is independent of $h$, and is the unique sign such that $\tilde{\epsilon} \equiv [\theta_C, \varphi] \pmod{p}$. Substituting above yields
	\begin{align*}
		\psi^G(y) &= \frac{|N|}{|T|} \sum_{\substack{h \in P \\ hyh^{-1} \in Q}} \tilde{\epsilon}\tilde{F}(\psi)(hyh^{-1}) = \tilde{\epsilon} \left(\frac{|C|}{|S_\varphi|} \sum_{\substack{h \in P \\ hyh^{-1} \in Q}} \tilde{F}(\psi)(hyh^{-1})\right)  \\
		&= \tilde{\epsilon} \tilde{F}(\psi)^{S}(y) = \tilde{\epsilon} F(\psi^G)(y),
	\end{align*}
	as we wanted to show. Notice that, if $h \in P$ is such that $x^h \in P_\theta$, then $\CB_N(x^h) \leq C$.
\end{proof}

The following result is the Strong Subnormalizer Conjecture in the case where $p \neq 2$ and $G$ has a normal $p$-complement.

\begin{thm}
	\label{SubConjNormalPComplement}
	Let $p$ be an odd prime, let $G$ have a normal $p$-complement $N$ and let $x$ be a $p$-element of $G$. Then, there exists a bijection $$F: \Irr^x(G) \to \Irr^x(\Sub_G(x))$$ such that $F(\chi)(1)_p = \chi(1)_p$ and $F(\chi)(x) = \epsilon \chi(x)$ for some $\epsilon \in \{1, -1\}$.
\end{thm}

\begin{proof}
	Write $G = NP$, where $P \in \Syl_p(G)$ contains $x$, and write $S = \Sub_G(x)$, $C = \CB_N(x)^S$. By \Cref{nonVanishingInvarConst}, if $\chi \in \Irr^x(G)$, there exists $\theta \in \Irr_{\langle x \rangle}(N)$ lying under $\chi$. Partition the set $\Irr_{\langle x \rangle}(N)$ according to the equivalence relation $\theta \sim \eta$ if there exists $g \in G$ such that $\eta = \theta^g$. Notice, that, in fact, such a $g$ can be taken in $P \subseteq \Sub_G(x)$. This partition induces, by the previous paragraph, a disjoint union 
	\begin{equation*}
		\Irr^x(G) = \bigsqcup_{\theta \in \Delta} \Irr^x(G\mid \theta),
	\end{equation*}
	where $\Delta$ is a complete set of representatives of this partition.
	
	Inducing an analogous partition (by $\Sub_G(x)$-conjugation) on the set $\Irr_{\langle x \rangle}(C)$, we obtain a disjoint union
	\begin{equation*}
		\Irr^x(\Sub_G(x)) = \bigsqcup_{\varphi \in \Xi} \Irr^x(\Sub_G(x) \mid \varphi).
	\end{equation*}
	If $\theta_1, \theta_2$ belong to the same equivalence class in $\Delta$, there exists $g \in P$ such that $\theta_2 = \theta_1^g$. In particular, $\theta_1^g$ is $\langle x \rangle$-invariant, meaning $gxg^{-1} \in P_{\theta_1} = P_{f_x(\theta_1)}$ by \Cref{inertiaGroupLemma}. Then, $f_x(\theta_1)^g$ is $\langle x \rangle$-invariant. Since this is an irreducible constituent of $(\theta_2)_C$ with multiplicity not divisible by $p$, the uniqueness of $f_x(\theta_2)$ implies $f_x(\theta_2) = f_x(\theta_1)^g$. The same argument done in reverse shows that, if $\varphi_1, \varphi_2$ belong to the same equivalence class in $\Xi$, then so do their correspondents via $f_x$.
	
	By the bijectivity in \Cref{correspondence} and the previous paragraph, we can thus consider $\Xi = \{f_x(\theta) \mid \theta \in \Delta\}$ and we can construct our map $F$ from maps $\Irr^x(G \mid \theta) \to \Irr^x(S \mid f_x(\theta))$ satisfying the conditions we want. The result thus follows from \Cref{normalPComplement}.
\end{proof}

We remark that our proof works verbatim for $p = 2$, provided we impose the condition $\langle x \rangle \unlhd P$. As mentioned before, this is equivalent, in this case, to imposing $\langle x \rangle \unlhd \Sub_G(x)$. In this particular setting, the sign $\epsilon$ that would then appear is the unique sign such that $\widehat{\theta}(x) = \epsilon \pi_x(\theta)(1)$.

\section{Groups of $p$-length 1}
\label{Section 4}

We now generalize the results of the last section to $p$-solvable groups which have $p$-length one. Under this more general hypothesis, we are still able to concisely describe the subnormalizer of a $p$-element $x$. Notice how this result agrees with \cite[Proposition 2.12]{GunterSubnormalizer} for $p$-solvable groups.

\begin{lem}
	Let $G$ be a $p$-solvable group of $p$-length $1$ and let $N = \OB_{p'}(G)$. Let $x \in G$ be a $p$-element and let $P \in \Syl_p(G)$ contain $x$. Then, $\Sub_G(x) = \langle \NB_G(P), \CB_N(x) \rangle$.
\end{lem}

\begin{proof}
	Let $Q \in \Syl_p(G)$ contain $x$. By the Frattini argument, since $NP \unlhd G$, we may write $G = N \NB_G(P)$. Then, there exists $n \in N$ such that $Q = P^n$ and $x, x^n \in Q$. Consequently, $[x, n] \in Q \cap N = 1$, meaning $n \in \CB_N(x)$.
	
	By \cite[Proposition 2.6]{GunterSubnormalizer}, we have $$\Sub_G(x) = \langle \NB_G(Q) \mid x \in Q \in \Syl_p(G) \rangle \subseteq \langle \NB_G(P), \CB_N(x) \rangle,$$ by what we have seen above. The other inclusion follows from the definition of $\Sub_G(x)$.
\end{proof}

We now prove the result needed for Theorem B, emulating what we did in the normal $p$-complement case. Notice that, if $h \in S$, then $\CB_N(x^h) \subseteq N \cap S$, just as had happened then.

\begin{thm}
	\label{ThmBKey}
	Let $p$ be an odd prime, $G$ a $p$-solvable group of $p$-length one, $N = \OB_{p'}(G)$, $x \in G$ a $p$-element and let $S = \Sub_G(x)$. Let $\theta \in \Irr_{\langle x \rangle}(N)$ and let $f_x(\theta) = \varphi \in \Irr_{\langle x \rangle}(S \cap N)$. Then, there exists a bijection $$F: \Irr^x(G \mid \theta) \to \Irr^x(S \mid \varphi)$$ such that $F(\chi)(1)_p = \chi(1)_p$ and $F(\chi)(y) = \epsilon \chi(y)$, for all $y \in P_\theta$ such that $\CB_N(y) \leq N \cap S$, where $\epsilon \in \{1, -1\}$ is the unique sign such that $[\theta_{S \cap N}, \varphi] \equiv \epsilon \pmod{p}$. In particular, this holds for all $y \in P_\theta$ which are $S$-conjugate to $x$.
\end{thm}

\begin{proof}
	Write $T = G_\theta$, let $Q \in \Syl_p(T)$ contain $x$ and let $P \in \Syl_p(G)$ contain $Q$. Since $G$ has $p$-length one, then $NP \unlhd G$. Thus, it follows that $G = N\NB_G(P)$, by the Frattini argument. Since $\NB_G(P) \subseteq S$, then $S = (S \cap N)\NB_G(P)$. Also, $T \cap S = (S\cap N)(\NB_G(P))_\theta = (S\cap N)(\NB_G(P))_\varphi = S_\varphi$, using \Cref{inertiaGroupLemma} with $L = 1, K = N, H = S$ and $A = \NB_G(P)$.
	
	Notice that $\CB_N(Q) \subseteq S \cap N$ and that $\varphi$ is a $Q$-invariant irreducible constituent of $\theta_{S \cap N}$ which has multiplicity not divisible by $p$. Then, $f_Q(\theta) = \varphi$. In particular, the $Q$-Glauberman correspondents of $\theta$ and $\varphi$ coincide.
	
	Now, from the fact that $NP \unlhd G$, it follows that $NQ \unlhd T$ and that $(S \cap N)Q \unlhd S_\varphi$. Let $\theta^*$ be the $Q$-Glauberman correspondent of $\theta$ (and $\varphi$). Then, $\theta$, $\varphi$ and $\theta^*$ have canonical extensions $\widehat{\theta}$, $\widehat{\varphi}$, $\widehat{\theta^*}$ to $\Irr(NQ), \Irr((S \cap N)Q), \Irr(\CB_N(Q)Q)$, respectively. 
	
	By \Cref{TurullLemma}, there exist character triple isomorphisms 
	\begin{equation*}
		(T, N, \theta) \cong (\NB_T(Q), \CB_N(Q), \theta^*) \text{ and } (S_\varphi, S \cap N, \varphi) \cong (\NB_{S_\varphi}(Q), \CB_{S \cap N}(Q), \theta^*).
	\end{equation*} 
	But notice how, since $S_\varphi = T \cap S$ and $\NB_G(Q) \subseteq S$, then $\NB_{S_\varphi}(Q) = \NB_T(Q)$; also, for similar reasons, $\CB_{S \cap N}(Q) = \CB_N(Q)$. Since $p$ is odd, these isomorphisms send $\widehat{\theta}$ and $\widehat{\varphi}$ to $\widehat{\theta^*}, \widehat{\varphi^*}$, respectively, by \Cref{TurullLemma}. Composing both, we get a character triple isomorphism $(T, N, \theta) \cong (S_\varphi, S \cap N, \varphi)$ that sends $\widehat{\theta}$ to $\widehat{\varphi}$.
	
	Denote the character triple isomorphism constructed above by $\dagger$ and let $\pi_y$ denote the $\langle y \rangle$-Glauberman correspondence, where $y \in Q$ is such that $\CB_N(y) \leq N \cap S$. From \Cref{inertiaGroupLemma}, it follows that $\varphi = f_y(\theta)$. Since character triple isomorphisms preserve character ratios, we have that $\chi^\dagger(1)_p = \chi(1)_p$ for all $\chi \in \Irr(T \mid \theta)$, using that $N$ is a $p'$-group. 
	
	For the values on $y$, $\widehat{\theta}_{N\langle y \rangle}$ is the canonical extension of $\theta$ to $N \langle y \rangle$ (similarly for $\varphi$). Hence, $\widehat{\theta}(y) = \epsilon_1 \pi_y(\theta)(1)$, where $\epsilon_1$ is the unique sign congruent to $[\theta_{\CB_N(y)}, \pi_y(\theta)]$ modulo $p$, by \cite[Theorem 13.6]{CTFG}. Since the $\langle y \rangle$-Glauberman correspondents of $\theta$ and $\varphi$ coincide, it follows that $\widehat{\varphi}(y) = \epsilon_2 \pi_y(\theta)(1)$, where $\epsilon_2$ is the unique sign such that $[\varphi_{\CB_N(x)}, \pi_y(\theta)]$ modulo $p$. By \Cref{correspondence}, $\epsilon := \epsilon_1\epsilon_2 \equiv [\theta_{S \cap N}, \varphi] \pmod{p}$.
	
	Now, if $\chi \in \Irr(T \mid \theta)$, write $y^\dagger$ for an element of $S_\varphi$ such that $y^\dagger (S \cap N) = (yN)^\dagger$. Notice that, since $\dagger$ is associated to the natural map $gN \mapsto g(S \cap N)$, we can take $y^\dagger = y$. Then, by the proof of \cite[Lemma 5.17]{CTMcKC}, we have
	\begin{equation*}
		\chi(y) = \frac{\widehat{\theta}(y)}{\widehat{\theta}^\dagger(y)} \chi^\dagger(y) = \frac{\widehat{\theta}(y)}{\widehat{\varphi}(y)} \chi^\dagger(y) = \epsilon \chi^\dagger(y).
	\end{equation*}
	
	Now, by the Clifford correspondence applied twice, we are able to define a bijection $F: \Irr(G \mid \theta) \to \Irr(S \mid \varphi)$ such that $F(\psi^G) = (\psi^\dagger)^S$. On the one hand, since $G = N\NB_G(P)$ and $N \cap \NB_G(P) = \CB_N(P)$, each element of $g$ can be written as a product $ab$, with $a \in N$, $b \in \NB_G(P)$ in exactly $|\CB_N(P)|$ ways. Consequently, we have:
	\begin{align*}
		\psi^G(y) &= \frac{1}{|T|} \sum_{\substack{g \in G \\ gyg^{-1} \in T}} \psi(gyg^{-1}) = \frac{1}{|T||\CB_N(P)|} \sum_{n \in N}\sum_{g \in \NB_G(P)} \psi^\circ(ngyg^{-1}n^{-1}) \\
		&= \frac{1}{|T:N||\CB_N(P)|}\sum_{\substack{g \in \NB_G(P) \\ gyg^{-1} \in T}} \psi(gyg^{-1}).
	\end{align*}
	Notice how, if $g \in \NB_G(P)$, since $y \in Q \subseteq P$, then $gyg^{-1} \in P$. Thus, $gyg^{-1} \in T$ if and only if $gyg^{-1} \in Q$. Since $\CB_N(y) \leq N \cap S$ and $\NB_G(P) \subseteq S$, we have $\CB_N(gyg^{-1}) \leq N \cap S$ as well. Then, by what we proved above, we have
	\begin{equation*}
		\psi^G(y) = \frac{\epsilon}{|T:N||\CB_N(P)|}\sum_{\substack{g \in \NB_G(P) \\ gyg^{-1} \in Q}} \psi^\dagger(gyg^{-1}).
	\end{equation*}
	At the same time, using the previous reasoning and the fact that $S = (N \cap S)\NB_G(P)$,
	\begin{align*}
		\epsilon(\psi^\dagger)^S(y) &= \frac{\epsilon}{|S_\varphi|}\sum_{\substack{g \in S \\ gyg^{-1} \in S_\varphi}} \psi^\dagger(gyg^{-1}) = \frac{\epsilon}{|S_\varphi||\CB_N(P)|} \sum_{n \in N \cap S} \sum_{g \in \NB_G(P)} (\psi^\dagger)^\circ(ngyg^{-1}n^{-1}) \\
		&= \frac{\epsilon}{|S_\varphi: S \cap N||\CB_N(P)|} \sum_{\substack{g \in \NB_G(P) \\ gyg^{-1} \in S_\varphi}} \psi^\dagger(gyg^{-1}).
	\end{align*}
	But notice that $|T:N| = |S_\varphi:S\cap N|$ and that $gyg^{-1} \in S_\varphi$ if and only if $gyg^{-1} \in S_\varphi \cap P$, for $g \in \NB_G(P)$. But $S_\varphi \cap P = T \cap S \cap P = Q$, since $P \subseteq S$. Comparing both equalities thus yields $\psi^G(y) = \epsilon F(\psi^G)$, as desired. The part on $p$-parts of degrees follows without much issue, just as in \Cref{normalPComplement}.
\end{proof}

Now, we are able to prove Theorem B, which we restate below.

\begin{thmB}
	Let $G$ be a $p$-solvable group of $p$-length one, where $p$ is an odd prime, let $x \in G$ be a $p$-element and let $S = \Sub_G(x)$. Then, there exists a bijection $$F: \Irr^x(G) \to \Irr^x(S)$$ such that $F(\chi)(1)_p = \chi(1)_p$ and $F(\chi)(x) = \epsilon_\chi\chi(x)$ for a sign $\epsilon_\chi \in \{1, -1\}$.
\end{thmB}

\begin{proof}
	Let $N = \OB_{p'}(G)$ and let $D = N \cap S$. We may partition the sets $\Irr_{\langle x \rangle}(N)$ and $\Irr_{\langle x \rangle}(D)$ by conjugation as in \Cref{SubConjNormalPComplement} and, using \Cref{nonVanishingInvarConst} just as we did then, we get partitions
	\begin{equation*}
		\Irr^x(G) = \bigsqcup_{\theta \in \Delta} \Irr^x(G \mid \theta)
	\end{equation*}
	and
	\begin{equation*}
		\Irr^x(S) = \bigsqcup_{\varphi \in \Xi} \Irr^x(S \mid \varphi),
	\end{equation*}
	where $\Delta, \Xi$ are complete sets of representatives of the partitions. By the same argument as in \Cref{SubConjNormalPComplement}, we can take $\Xi = \{f_x(\theta) \mid \theta \in \Delta\}$. The result thus follows from \Cref{ThmBKey}.
\end{proof}

\section{Proof of Theorem A}
\label{Section 5}

We now strive towards a proof of Theorem A, by removing the hypothesis on the $p$-length of $G$ and following similar guidelines as in \cite{PickyPSolvable}. We do this under the condition that $\langle x \rangle {\triangleleft}{\triangleleft} \Sub_G(x)$, which, as mentioned before, forces $\Sub_G(x) = S_G(x)$, the subnormalizer subset of $x$. This allows us to control inertia groups of characters corresponding under \Cref{relativeCorrespondence}.

Let us assume, for the moment, that the following result holds.

\begin{thm}
	\label{keyResult}
	Let $G$ be a finite group, $p$ and odd prime, $L \leq K \leq G$ such that $L, K \lhd G$ and $K/L$ is a $p'$-group. Suppose $x \in G$ is a $p$-element such that $\langle xL \rangle {\triangleleft}{\triangleleft} \Sub_{G/L}(xL)$. Writing $H/L = \Sub_{G/L}(xL)$, suppose that $G = KH$ and write $D = K \cap H$. Let $\theta \in \Irr_{\langle x \rangle}(K)$ and let $f_x(\theta) = \varphi \in \Irr_{\langle x \rangle}(D)$ be the correspondent of $\theta$ via \Cref{relativeCorrespondence}. Then, there exists a bijection
	\begin{equation*}
		F: \Irr(G \mid \theta) \to \Irr(H \mid \varphi)
	\end{equation*}
	such that $F(\chi)(1)_p = \chi(1)_p$ and $\chi(y) = \epsilon F(\chi)(y)$ for all $p$-elements $y$ such that $\Sub_{G/L}(yL) = \Sub_{G/L}(xL)$ and $\langle yL \rangle {\triangleleft}{\triangleleft} \Sub_{G/L}(yL)$, where $\epsilon$ is the unique sign such that $[\theta_D, \varphi] \equiv \epsilon \pmod{p}$.
\end{thm}

\begin{remark}
	Notice how $H = L\Sub_{G}(x)$ by \Cref{SubnormalizerProperties}. We have that $\CB_{K/L}(xL) \subseteq K/L$ and also, since $\CB_{G/L}(xL) \subseteq \Sub_{G/L}(xL)$, we have $\CB_{K/L}(x) \subseteq (K \cap H)/L = D/L$. Hence, the statement above is under the conditions of \Cref{relativeCorrespondence}, so that $f_x : \Irr_{\langle x \rangle}(K) \to \Irr_{\langle x \rangle}(D)$ is defined.
\end{remark}

Using \Cref{keyResult}, whose proof we postpone, we are able to prove Theorem A, which we restate below.

\begin{thmA}
	Let $G$ be a finite $p$-solvable group, where $p$ is an odd prime, and let $x \in G$ be a $p$-element such that $\langle x \rangle {\triangleleft}{\triangleleft} \Sub_G(x)$. Then, there exists a bijection
	\begin{equation*}
		F: \Irr^x(G) \to \Irr^x(\Sub_G(x))
	\end{equation*}
	such that $F(\chi)(1)_p = \chi(1)_p$ and $F(\chi)(x) = \epsilon_\chi \chi(x)$ for some $\epsilon_\chi \in \{1, -1\}$.
\end{thmA}

\begin{proof}
	We may assume that $\langle x \rangle$ is not subnormal in $G$, as otherwise, $G = \Sub_G(x)$ and the result is trivial. We now proceed by induction on $|G|$. Let $L_0 = \OB^{p'}(G)$ and define inductively $K_i = \OB^p(L_{i-1})$, $L_i = \OB^{p'}(K_i)$. By construction, $K_i/L_i$ is a $p'$-group for all $i$. Also, since $G$ is $p$-solvable, there exists some $j$ such that $L_j$ is a (possibly trivial) $p$-group. Then, $L_j \langle x \rangle$ is a $p$-group containing $x$, meaning $L_j\langle x \rangle \leq \Sub_G(x)$ by \Cref{SubnormalizerProperties}. In particular, $L_j \leq \Sub_G(x)$ and $G \neq L_j \Sub_G(x) = \Sub_G(x)$. Thus, there exists $k = \min \{1 \leq i \leq j \mid G \neq L_i \Sub_G(x)\}$.
	
	Let $Q \in \Syl_p(L_{k-1}\langle x \rangle)$ contain $x$. Then, by construction, $K_k Q = L_{k-1}\langle x \rangle$. If $k = 1$, then $L_{k-1} = L_0$ contains $x$ and $L_0 \unlhd G$. If $k \neq 1$, then $G = L_{k-1}\Sub_G(x)$ and, since $\langle x \rangle {\triangleleft}{\triangleleft} \Sub_G(x)$, it follows that $L_{k-1}\langle x \rangle {\triangleleft}{\triangleleft} L_{k-1}\Sub_G(x) = G$. Thus, in any case, $K_kQ {\triangleleft}{\triangleleft} G$. By \Cref{SubnormalizerProperties}, it follows that $G = K_k \Sub_G(x)$, since $x \in Q$. Then, $G = K_k H$, where $H = L_k \Sub_G(x)$. Since $H \neq G$ by the definition of $k$, by our inductive hypothesis, there exists a bijection $\hat{F} : \Irr^x(H) \to \Irr^x(\Sub_G(x))$ satisfying the conditions we want. So, it suffices to construct a suitable bijection $\Irr^x(G) \to \Irr^x(H)$.
	
	Partition the set $\Irr_{\langle x \rangle}(K_k)$ according to the equivalence relation $\theta \sim \eta$ if there exists $g \in G$ such that $\eta^g = \theta$ and let $\Delta$ be a complete set of representatives of this partition. By \Cref{nonVanishingInvarConst}, there is an induced decomposition
	\begin{equation*}
		\Irr^x(G) = \bigsqcup_{\theta \in \Delta} \Irr^x(G \mid \theta).
	\end{equation*}
	Doing the same to $\Irr_{\langle x \rangle}(D)$, where $D = K_k \cap H$ (using $H$-conjugation instead), we get another decomposition
	\begin{equation*}
		\Irr^x(H) = \bigsqcup_{\eta \in \Xi} \Irr^x(H \mid \eta).
	\end{equation*}
	
	Let $\theta_1, \theta_2 \in \Irr_{\langle x \rangle}(K_k)$ and suppose there exists $g \in G$ such that $\theta_2 = \theta_1^g$. Since $G = K_k \Sub_G(x)$, then we may assume that $g \in \Sub_G(x)$. Then, since both $\theta_1$ and $\theta_2$ are $\langle x \rangle$-invariant, it follows that $\theta_1^{gxg^{-1}} = \theta_2^{g^{-1}} = \theta_1$, meaning $gxg^{-1} \in \Sub_G(x)_{\theta_1}$. By \Cref{inertiaGroupLemma}, $f_x(\theta_1)^{gxg^{-1}} = f_x(\theta_1)$, which is equivalent to saying that $f_x(\theta_1)^g$ is $\langle x \rangle$-invariant. But it is an irreducible constituent of $(\theta_2)_D$ and $[(\theta_2)_D, f_x(\theta_1)^g] = [(\theta_1)_D, f_x(\theta_1)] \not \equiv 0 \pmod{p}$. By the uniqueness of \Cref{relativeCorrespondence}, $f_x(\theta_1)^g = f_x(\theta_2)$. 
	
	The same argument done in reverse shows that if $\varphi_1, \varphi_2 \in \Irr_{\langle x \rangle}(D)$ are $H$-conjugate, then so are their correspondents in $\Irr_{\langle x \rangle}(K_k)$. Consequently, we may take $\Xi = \{f_x(\theta) \mid \theta \in \Delta\}$. By \Cref{keyResult}, using the decompositions above we obtain a bijection $\Irr^x(G) \to \Irr^x(H)$ satisfying our desired conditions. This finishes the proof.
\end{proof}

So we are left with proving \Cref{keyResult}. We do so in a series of steps, following the general outlines of \cite{PickyPSolvable}, and dropping the extra hypotheses as we go. For our first step, we assume that $L \subseteq \ZB(G)$ and that $\theta$ is, in fact, $G$-invariant. 

The key observation for the proof of \Cref{Part1} is part (v) of \Cref{SubnormalizerProperties}, which allows us to replace $\Sub_G(x)$ by a normalizer $\NB_G(Q)$. This is crucial, in that it allows us to use \Cref{TurullLemma} and all the properties of the character triple isomorphism that it yields. Perhaps it could be the case that an analogous character triple isomorphism would hold in similar circumstances, without the hypothesis $\Sub_G(x) = S_G(x)$, as was also suggested by the proof of \Cref{ThmBKey}. As of writing, this seems unclear.

\begin{thm}
	\label{Part1}
	Let $G$ be a finite group, $p$ and odd prime, $L \leq K \leq G$ such that $L, K \lhd G$, $K/L$ is a $p'$-group and $L \subseteq \ZB(G)$. Suppose $x \in G$ is a $p$-element such that $\langle xL \rangle {\triangleleft}{\triangleleft} \Sub_{G/L}(xL)$. Writing $H/L = \Sub_{G/L}(xL)$, suppose also that $G = KH$ and write $D = K \cap H$. Suppose $\theta \in \Irr(K)$ is $G$-invariant and let $f_x(\theta) = \varphi \in \Irr(D)$ be the correspondent of $\theta$ as in \Cref{relativeCorrespondence}. Then, there exists a bijection
	\begin{equation*}
		F: \Irr(G \mid \theta) \to \Irr(H \mid \varphi)
	\end{equation*}
	such that $F(\chi)(1)_p = \chi(1)_p$ and $\chi(y) = \epsilon F(\chi)(y)$ for all $p$-elements $y$ such that $\langle yL \rangle {\triangleleft}{\triangleleft} \Sub_{G/L}(yL)$ and $\Sub_{G/L}(yL) = H/L$, where $\epsilon$ is the unique sign such that $[\theta_D, \varphi] \equiv \epsilon \pmod{p}$.
\end{thm}

\begin{proof}
	Since $L \subseteq \ZB(G)$, we may write $K = X \times L_p$, where $L_p \in \Syl_p(L)$ and $X$ is a normal $p$-complement of $K$. This also allows us to write $\theta = \theta_{p'} \times \theta_p$, with $\theta_{p'} \in \Irr(X)$ and $\theta_p \in \Irr(L_p)$. Since $\theta$ is $G$-invariant and $L \subseteq \ZB(G)$, $\theta_{p'}$ is also $G$-invariant. The centrality of $L$ also implies that $H = \Sub_G(x)$ by \Cref{SubnormalizerProperties} and that $\langle x \rangle \unlhd L \langle x \rangle$. At the same time, since $\langle xL \rangle {\triangleleft}{\triangleleft} H/L$, it follows that $L \langle x \rangle {\triangleleft}{\triangleleft} H$. Combining the two, $\langle x \rangle {\triangleleft}{\triangleleft} H$, meaning $H = \Sub_G(x) = S_G(x)$.
	
	Let $Q$ be the intersection of the Sylow $p$-subgroups containing $x$. By \Cref{SubnormalizerProperties}, $H = \NB_G(Q)$. Then, since $G = KH = XH$, we have $G = X\NB_G(Q)$; in particular, $XQ \unlhd G$. Notice how $L_p \subseteq Q$, since it is a normal $p$-subgroup of $G$. Also, $X \cap H = \NB_X(Q) = \CB_X(Q)$. At the same time, however, since $H = \Sub_G(x)$, $\CB_G(x) \subseteq H$, meaning $\CB_X(x) \subseteq X \cap H = \CB_X(Q)$. As $x \in Q$, we have the equality $\CB_X(Q) = \CB_X(x)$.
	
	Let $\pi_Q$ denote the $Q$-Glauberman correspondence. Since $\theta_{p'}$ is $G$-invariant, both $\pi_Q(\theta_{p'})$ and $\pi_x(\theta_{p'})$ are defined. By the equality of the previous paragraph, it follows that $\pi_Q(\theta_{p'}) = \pi_x(\theta_{p'})$, since $x \in Q$.
	
	Now, since $L \leq H$, then $D = H \cap K = (H \cap X) \times L_p$. This allows us to write $\varphi = \varphi_{p'} \times \varphi_p$, just as we did with $\theta$. Since $\pi_x(\theta_{p'}) \times \theta_p$ is an $\langle x \rangle$-invariant irreducible constituent of $\theta_D$ such that $[\theta_D, \pi_x(\theta_{p'}) \times \theta_p] = [(\theta_{p'})_{H \cap X}, \pi_x(\theta_{p'})] \not \equiv 0 \pmod{p}$, it follows by \Cref{relativeCorrespondence} that $\varphi = \pi_x(\theta_{p'}) \times \theta_p$; equivalently, that $\varphi_{p'} = \pi_x(\theta_{p'})$ and $\varphi_p = \theta_p$.
	
	By \Cref{inertiaGroupLemma}, using that $\langle x \rangle {\triangleleft}{\triangleleft} H$ and that $\theta$ is $G$-invariant, it follows that $\varphi$ is $H$-invariant, and so is $\pi_x(\theta_{p'})$. Then, $\theta_{p'}$ and $\pi_x(\theta_{p'})$ have canonical extensions $\tau_Q$ and $\nu_Q$ to $XQ, (X \cap H)Q$, respectively. In fact, since $X \cap H = \CB_X(Q)$, we have $\nu_Q = \pi_x(\theta_{p'}) \times 1_Q$.
	
	By \Cref{TurullLemma}, there exists a character triple isomorphism from $(G, X, \theta_{p'})$ to $(H, X \cap H, \pi_x(\theta_{p'}))$ which, as $p \neq 2$, sends $\tau_Q$ to $\nu_Q$. Denote this character triple isomorphism by $\dagger$. Arguing as in \cite[Theorem 3.1]{PickyPSolvable}, we see that $\dagger$ induces a bijection from $\Irr(G\mid\theta)$ to $\Irr(H \mid \varphi)$. We claim that this map satisfies our desired properties. First, if $\chi \in \Irr(G \mid \theta)$, then $\chi(1)_p = \chi^\dagger(1)_p$, since $X$ is a group of $p'$-order and character triple isomorphisms preserve degree ratios. 
	
	Now let $y$ be a $p$-element such that $\langle yL \rangle {\triangleleft}{\triangleleft} \Sub_{G/L}(yL)$ and $\Sub_{G/L}(yL) = H/L$. Just as was the case with $x$, it follows that $\langle y \rangle {\triangleleft}{\triangleleft} H$ and that $H = \Sub_G(y)$. Consequently, $H = \NB_G(P)$, where $P$ is the intersection of the Sylow $p$-subgroups containing $y$, and the $P$-Glauberman correspondent of $\theta_{p'}$ coincides with its $\langle y \rangle$-Glauberman correspondent. Let $\tau_P$, $\nu_P$ be the canonical extensions of $\theta_{p'}$, $\pi_y(\theta_{p'})$ to $XP$, $(H\cap X)P$, respectively. By \Cref{TurullLemma}, $\pi_y(\theta_{p'}) = \pi_x(\theta_{p'})$ and $\tau_P^\dagger = \nu_P$. Just as was the case with $Q$, $\nu_P = \pi_x(\theta_{p'}) \times 1_P$.
	
	Notice that ${\tau_P}_{X\langle y \rangle}$ is the canonical extension of $\theta_{p'}$ to $X \langle y \rangle$. By \cite[Theorem 13.6]{CTFG}, it follows that
	\begin{equation*}
		\tau_P(y) = \epsilon \pi_x(\theta_{p'})(1),
	\end{equation*}
	where, by \cite[Theorem 13.14]{CTFG}, $\epsilon$ is the unique sign such that $\epsilon \equiv [({\theta_{p'}})_{X\cap H}, \pi_x(\theta_{p'})] = [\theta_D, \varphi] \pmod{p}$. By our previous description of $\nu_P$,
	\begin{equation*}
		\nu_P(y) = \pi_x(\theta_{p'})(1).
	\end{equation*}
	Just as was done in \Cref{ThmBKey}, we have
	\begin{equation*}
		\chi(y) = \frac{\tau_P(y)}{\tau_P^\dagger(y)} \chi^\dagger(y) = \frac{\tau_P(y)}{\nu_P(y)} \chi^\dagger(y) = \epsilon \chi^\dagger(y),
	\end{equation*}
	as we wanted to show.
\end{proof}

We now start removing some of the hypotheses from the previous result, beginning with the one on the centrality of $L$. This is done in very similar fashion to \cite[Theorem 3.2]{PickyPSolvable}, with several small adaptations throughout.

\begin{thm}
	\label{Part2}
	Let $G$ be a finite group, $p$ and odd prime, $L \leq K \leq G$ such that $L, K \lhd G$ and $K/L$ is a $p'$-group. Suppose $x \in G$ is a $p$-element such that $\langle xL \rangle {\triangleleft}{\triangleleft} \Sub_{G/L}(xL)$. Writing $H/L = \Sub_{G/L}(xL)$, suppose also that $G = KH$ and write $D = K \cap H$. Suppose $\theta \in \Irr(K)$ is $G$-invariant and let $f_x(\theta) = \varphi \in \Irr(D)$ be the correspondent of $\theta$ as in \Cref{relativeCorrespondence}. Suppose also that there exists $\xi \in \Irr(L)$ under $\varphi$ which is $G$-invariant. Then, there exists a bijection
	\begin{equation*}
		F: \Irr(G \mid \theta) \to \Irr(H \mid \varphi)
	\end{equation*}
	such that $F(\chi)(1)_p = \chi(1)_p$ and $\chi(y) = \epsilon F(\chi)(y)$ for all $p$-elements $y$ such that $\langle yL \rangle {\triangleleft}{\triangleleft} \Sub_{G/L}(yL)$ and $\Sub_{G/L}(yL) = H/L$, where $\epsilon$ is the unique sign such that $[\theta_D, \varphi] \equiv \epsilon \pmod{p}$.
\end{thm}

\begin{proof}
	Let $(G^*, L^*, \xi^*)$ be a character triple \emph{strongly} isomorphic to $(G, L, \xi)$ with $L^* \subseteq \ZB(G^*)$ (see \cite[Problem 5.4]{CTMcKC}, for instance). Write (simultaneously, with some abuse of notation) $*$ for the isomorphism $G/L \to G^*/L^*$ and bijections of characters $\Irr(J \mid \xi) \to \Irr(J^* \mid \xi^*)$, where $L \leq J \leq G$ and $J/L \cong J^*/L^*$. Let $x^*$ be a $p$-element of $G^*$ such that $x^*L^* = (xL)^*$.
	
	Notice how $(\Sub_{G/L}(xL))^* = \Sub_{G^*/L^*}(x^*L^*)$, since the normalizers of the Sylow $p$-subgroups of $G/L$ containing $xL$ get mapped onto those of the Sylow $p$-subgroups of $G^*/L^*$ containing $x^*L^*$ and vice-versa. Consequently, $\langle x^*L^* \rangle$ is subnormal in $\Sub_{G^*/L^*}(x^*L^*)$. Also, $(\theta^*)^{(gL)^*} = (\theta^{gL})^* = \theta^*$ for all $g \in G$. The same can be done for $\varphi$ and $h \in H$ ($\varphi$ is $H$-invariant by \Cref{inertiaGroupLemma}). From this, it follows that $\varphi^* = f_{x^*}(\theta^*)$ and that $G^*$ satisfies the hypotheses of \Cref{Part1}. 
	
	We then have a bijection $\tilde{F}:\Irr(G^* \mid \theta^*) \to \Irr(H^* \mid \varphi^*)$ that satisfies the conditions $\chi(1)_p = \tilde{F}(\chi)(1)_p$ and $\chi(y^*) = \epsilon \tilde{F}(\chi)(y^*)$ for all $p$-elements $y^*$ such that $\langle y^*L^* \rangle {\triangleleft}{\triangleleft} \Sub_{G^*/L^*}(y^*L^*)$ and $\Sub_{G^*/L^*}(y^*L^*) = H^*/L^*$, where $\epsilon$ is as in the statement of the theorem (since $[\theta_D, \varphi] = [\theta^*_{D^*}, \varphi^*]$, by the properties of character triple isomorphisms). 
	
	By the definitions of character triple isomorphisms, $*$ induces bijections $\Irr(G \mid \theta) \to \Irr(G^* \mid \theta^*)$ and $\Irr(H \mid \varphi) \to \Irr(H^* \mid \varphi^*)$. Then, using $*$ and $\tilde{F}$, we can obtain a bijection $F: \Irr(G \mid \theta) \to \Irr(H \mid \varphi)$, defined by $F(\chi)^* = \tilde{F}(\chi^*)$, which we claim to be the bijection from the statement.
	
	First, since character triple isomorphisms preserve character degree ratios, we obtain
	\begin{equation*}
		\frac{\chi(1)_p}{\xi(1)_p} = \frac{\chi^*(1)_p}{\xi^*(1)_p}= \frac{\tilde{F}(\chi^*)(1)_p}{\xi^*(1)_p}= \frac{F(\chi)^*(1)_p}{\xi^*(1)_p} = \frac{F(\chi)(1)_p}{\xi(1)_p},
	\end{equation*}
	from which $\chi(1)_p = F(\chi)(1)_p$. 
	
	For the second property, let $y$ be a $p$-element such that $\langle yL \rangle {\triangleleft}{\triangleleft} \Sub_{G/L}(yL)$ and $\Sub_{G/L}(yL) = H/L$. Let $y^*$ be a $p$-element of $G^*$ such that $y^*L^* = (yL)^*$. Just as was the case with $x$, $H^*/L^* = (\Sub_{G/L}(yL))^* = \Sub_{G^*/L^*}(y^*L^*)$ and $\langle y^*L^* \rangle {\triangleleft}{\triangleleft} H^*/L^*$. Consequently, $\chi(y^*) = \epsilon \tilde{F}(y^*)$ for all $\chi \in \Irr(G^* \mid \theta^*)$.
	
	Let $\tau$ be an extension of $\xi$ to $L\langle y \rangle$. By the proof of \cite[Lemma 5.17]{CTMcKC}, we get
	\begin{equation*}
		\chi(y) = \frac{\tau(y)}{\tau^*(y^*)}\chi^*(y^*) = \epsilon \frac{\tau(y)}{\tau^*(y^*)} \tilde{F}(\chi^*)(y^*),
	\end{equation*}
	with $\epsilon$ as in the statement. By the same argument,
	\begin{equation*}
		F(\chi)(y) = \frac{\tau(y)}{\tau^*(y^*)} F(\chi)^*(y^*) = \frac{\tau(y)}{\tau^*(y^*)} \tilde{F}(\chi^*)(y^*),
	\end{equation*} 
	meaning $\chi(y) = \epsilon F(\chi)(y)$, as desired.
\end{proof}

Next, we remove the condition on the existence of a $G$-invariant $\xi$ under $\varphi$.

\begin{thm}
	\label{Part3}
	Let $G$ be a finite group, $p$ and odd prime, $L \leq K \leq G$ such that $L, K \lhd G$ and $K/L$ is a $p'$-group. Suppose $x \in G$ is a $p$-element such that $\langle xL \rangle {\triangleleft}{\triangleleft} \Sub_{G/L}(xL)$. Writing $H/L = \Sub_{G/L}(xL)$, suppose also that $G = KH$ and write $D = K \cap H$. Suppose $\theta \in \Irr(K)$ is $G$-invariant and let $f_x(\theta) = \varphi \in \Irr(D)$ be the correspondent of $\theta$ as in \Cref{relativeCorrespondence}. Then, there exists a bijection
	\begin{equation*}
		F: \Irr(G \mid \theta) \to \Irr(H \mid \varphi)
	\end{equation*}
	such that $F(\chi)(1)_p = \chi(1)_p$ and $\chi(y) = \epsilon F(\chi)(y)$ for all $p$-elements $y$ such that $\langle yL \rangle {\triangleleft}{\triangleleft} \Sub_{G/L}(yL)$ and $\Sub_{G/L}(yL) = H/L$, where $\epsilon$ is the unique sign such that $[\theta_D, \varphi] \equiv \epsilon \pmod{p}$.
\end{thm}

\begin{proof}
	By \Cref{ElementaryCoprimeActionLemma}, there exists some $\xi \in \Irr_{\langle x \rangle}(L)$ lying under $\varphi$. Let $G_\xi$ be the stabilizer of $\xi$ in $G$ and let $g \in G$. Since $\theta$ is $G$-invariant, $\xi^g$ is an irreducible constituent of $\theta_L$. Consequently, there exists $k \in K$ such that $\xi^g = \xi^k$. Hence, $g = ky$, where $y \in G_\xi$. As $g$ was arbitrary, it follows that $G = KG_\xi$. Using the same logic with the fact that $\varphi$ is $H$-invariant (by \Cref{inertiaGroupLemma}) we get $H = DH_\xi$. Thus, $G = KH_\xi$, meaning $G_\xi = K_\xi H_\xi$; it also follows that $D_\xi = K_\xi \cap H_\xi$. Also notice that, since $\Sub_{G_\xi/L}(xL) \leq \Sub_{G/L}(xL)$, we have $\langle xL \rangle {\triangleleft}{\triangleleft} \Sub_{G_\xi/L}(xL)$.
	
	Let $\theta_\xi$ and $\varphi_\xi$ be the Clifford correspondents of $\theta$ and $\varphi$ over $\xi$, respectively. Notice how $\CB_{K_\xi/L}(x) \subseteq K_\xi /L \cap H/L = D_\xi/L$, from which $f_x(\theta_\xi)$ is defined. Arguing as in \cite[Theorem 3.4]{PickyPSolvable}, we see that $f_x(\theta_\xi) = \varphi_\xi$. Consequently, $G_\xi$ (with the appropriate subgroups) is under the conditions of \Cref{Part2}. By that result, we obtain a bijection $\tilde{F}: \Irr(G_\xi \mid \theta_\xi) \to \Irr(H_\xi \mid \varphi_\xi)$ satisfying $\chi(1)_p = \tilde{F}(\chi)(1)_p$ and $\chi(y) = \epsilon \tilde{F}(\chi)(y)$ for all $p$-elements $y \in G_\xi$ such that $\langle yL \rangle {\triangleleft}{\triangleleft} \Sub_{G_\xi/L}(yL)$ and $\Sub_{G_\xi/L}(yL) = H_\xi/L$, where $\epsilon \equiv [(\theta_\xi)_{D_\xi}, \varphi_\xi] \pmod{p}$.
	
	By the Clifford correspondence, induction of characters defines a bijection $\Irr(G_\xi \mid \xi) \to \Irr(G \mid \xi)$. Following the argument of \cite[Theorem 3.4]{PickyPSolvable}, this bijection restricts to a bijection $\Irr(G_\xi \mid \theta_\xi) \to \Irr(G \mid \theta)$. Analogously, we get a bijection $\Irr(H_\xi \mid \varphi_\xi) \to \Irr(H \mid \varphi)$. So, we can combine these with our bijection of the previous paragraph to get a bijection $F: \Irr(G \mid \theta) \to \Irr(H \mid \varphi)$ defined by the property that $F(\psi^G) = \tilde{F}(\psi)^H$, where $\psi \in \Irr(G_{\xi} \mid \theta_\xi)$. We now have to prove that this is our desired bijection. 
	
	For the first part, notice how
	\begin{equation*}
		F(\psi^G)(1)_p = \tilde{F}(\psi)^H(1)_p = |H:H_\xi|_p\psi(1)_p = \frac{|H:H_\xi|_p}{|G:G_\xi|_p}\psi^G(1)_p.
	\end{equation*}
	From the fact that $G = KH_\xi$ and that $K/L$ is a $p'$-group, we get $|G|_p = |H_\xi|_p$; in particular, it also holds that $|G|_p = |G_\xi|_p$. Then, the previous equation yields $F(\psi^G)(1)_p = \psi^G(1)_p$. 
	
	Now let $y$ be a $p$-element of $G$ such that $\langle yL \rangle {\triangleleft}{\triangleleft} \Sub_{G/L}(yL)$ and $\Sub_{G/L}(yL) = H/L$. Then $y \in H$, meaning $\varphi$ is $\langle y \rangle$-invariant. By \Cref{ElementaryCoprimeActionLemma} and Clifford's theorem, there exists $c \in D$ such that $\xi^c$ is $\langle y \rangle$-invariant. 
	
	Notice how $(G_\xi)^c = G_{\xi^c}$ (and analogously for $K_\xi, H_\xi, D_\xi$) and how $(\theta_\xi)^c \in \Irr(K_{\xi^c} \mid \xi^c)$ is such that $(\theta_\xi^c)^K = \theta^c = \theta$. Consequently, $\theta_\xi^c = \theta_{\xi^c}$ is the Clifford correspondent of $\theta$ over $\xi^c$. Similarly, $\varphi_\xi^c = \varphi_{\xi^c}$ is the Clifford correspondent of $\varphi$ over $\xi^c$. Then, conjugation by $c$ induces bijections $\Irr(G_\xi \mid \theta_\xi) \to \Irr(G_{\xi^c} \mid \theta_{\xi^c})$ and $\Irr(H_\xi \mid \varphi_\xi) \to \Irr(H_{\xi^c} \mid \varphi_{\xi^c})$. So we are able to construct a new bijection $\hat{F}: \Irr(G_{\xi^c} \mid \theta_{\xi^c}) \to \Irr(H_{\xi^c} \mid \varphi_{\xi^c})$ given by $\hat{F}(\psi^c) = \tilde{F}(\psi)^c$. We claim that this new function also satisfies the properties we had for $\tilde{F}$.
	
	First, for the $p$-parts of the degrees, we have 
	\begin{equation*}
		\hat{F}(\psi^c)(1)_p = \tilde{F}(\psi)^c(1)_p = \tilde{F}(\psi)(1)_p = \psi^c(1)_p,
	\end{equation*}
	for all $\psi \in \Irr(G_\xi \mid \theta_\xi)$. Thus, $\hat{F}(\chi)(1)_p = \chi(1)_p$ for all $\chi \in \Irr(G_{\xi^c} \mid \theta_{\xi^c})$.
	
	Now let $z \in G_{\xi^c}$ be a $p$-element such that $\langle zL \rangle {\triangleleft}{\triangleleft} \Sub_{G_{\xi^c}/L}(zL)$ and $\Sub_{G_{\xi^c}/L}(zL) = H_{\xi^c}/L$. Then, we may write $z = \tilde{z}^c$ for some $p$-element $\tilde{z} \in G_\xi$. Also, since $\langle zL \rangle {\triangleleft}{\triangleleft} \Sub_{G_{\xi^c}/L}(zL)$, then $\langle \tilde{z}L \rangle {\triangleleft}{\triangleleft} (\Sub_{G_{\xi^c}/L}(zL))^{c^{-1}} = \Sub_{G_\xi/L}(\tilde{z}L)$. And since $\Sub_{G_{\xi^c}/L}(zL) = H_{\xi^c}/L$, then $\Sub_{G_\xi/L}(\tilde{z}L) = H_\xi/L$. Consequently, given $\psi \in \Irr(G_\xi \mid \theta_\xi)$, we have
	\begin{equation*}
		\psi^c(z) = \psi(\tilde{z}) = \epsilon \tilde{F}(\psi)(\tilde{z}) = \epsilon \tilde{F}(\psi)^c(z) = \epsilon \hat{F}(\psi^c)(z),
	\end{equation*} 
	as wanted. Also, $F(\psi^G) = \tilde{F}(\psi)^H = (\tilde{F}(\psi)^c)^H = \hat{F}(\psi^c)^H$. We will thus use $\hat{F}$ to calculate $F(\psi^G)(y)$.
	
	Let $\Delta = \{k \in K \mid kyk^{-1} \in G_{\xi^c}\}$ and let $\Omega = \{d \in D \mid dyd^{-1} \in H_{\xi^c}\}$. We claim that $\Delta = K_{\xi^c}\Omega$ as a set. For one inclusion, if $k \in K_{\xi^c}$ and $d \in \Omega$, then $dyd^{-1} \in H_{\xi^c}$ and $k \in K_{\xi^c}$, from which $kdyd^{-1}k^{-1} \in G_{\xi^c}$. Also, since $D \leq K$, we have $kd \in K$.
	
	For the other inclusion, let $k \in \Delta$, so that $kyk^{-1} \in G_{\xi^c}$. Then, there exists $P \in \Syl_p(G_{\xi^c})$ such that $kyk^{-1} \in K_{\xi^c}P$. We also know that $K_{\xi^c}\langle y \rangle = G_{\xi^c} \cap (K\langle y \rangle)$, since $y \in G_{\xi^c}$. But $G = KH$ and $H = L \Sub_G(y)$. Since $\langle yL \rangle {\triangleleft}{\triangleleft} H/L$, then $K\langle y \rangle {\triangleleft}{\triangleleft} G$. Putting everything together yields $K_{\xi^c} \langle y \rangle {\triangleleft}{\triangleleft} G_{\xi^c}$. Then, $(K_{\xi^c} \langle y \rangle)/K_{\xi^c}$ is a subnormal $p$-subgroup of $G_{\xi^c}/K_{\xi^c}$, meaning it is contained in every Sylow $p$-subgroup of this quotient (by, for instance, \cite[Problem 2A.1]{FGT}). In particular, $y \in K_{\xi^c}P$. Thus, there exists $Q \in \Syl_p(K_{\xi^c}P)$ such that $y \in Q$. By order considerations, it follows that $K_{\xi^c}P = K_{\xi^c}Q$.
	
	As $K_{\xi^c}/L$ is a group of $p'$-order, $(LQ)/L$ is a Sylow $p$-subgroup of $(K_{\xi^c}Q)/L$. Thus, there exists $v \in K_{\xi^c}$ such that $vkyk^{-1}v^{-1} \in LQ$. This means that $yL \in (LQ)/L$ and $yL \in (LQ^{vk})/L$. Then, there exists $sL \in \Sub_{G/L}(yL)$ such that $vksL \in \NB_{G/L}(LQ/L) \leq \Sub_{G/L}(yL)$. This implies $vkL \in \Sub_{G/L}(yL) = H/L$, from which $vkL \in K/L \cap H/L = D/L$ and $vk \in D$. Finally, since $v \in K_{\xi^c}$ and $k \in \Delta$, we have $v(kyk^{-1})v^{-1} \in H_{\xi^c}$. By definition, then, $vk \in \Omega$ and $k = v^{-1}(vk) \in K_{\xi^c}\Omega$.
	
	Notice how $K_{\xi^c} \cap \Omega = \{d \in D_{\xi^c} \mid dyd^{-1} \in H_{\xi^c}\} = D_{\xi^c}$. Define $\Psi: K_{\xi^c} \times \Omega \to \Delta$ by $\Psi(k, d) = kd$. Fix $a \in \Delta$ and write $a = \Psi(k, d)$. For each $b \in D_{\xi^c}$, we have $b^{-1}dyd^{-1}b \in H_{\xi^c}$, meaning $b^{-1}d \in \Omega$, and $\Psi(kb, b^{-1}d) = a$. At the same time, if $a = \Psi(k', d')$, then $k^{-1}k' = dd'^{-1} \in K_{\xi^c} \cap D = D_{\xi^c}$. This implies that $a$ can be written in exactly $|D_{\xi^c}|$ ways as a product in $K_{\xi^c}\Omega$.
	
	By \cite[Lemma 3.3]{PickyPSolvable}, using that $G = KG_{\xi^c}$, we get
	\begin{equation*}
		(\psi^c)^G(y) = \frac{1}{|K_{\xi^c}|}\sum_{k \in \Delta} \psi^c(kyk^{-1}).
	\end{equation*}
	By what we did in the previous paragraph, we can write this as
	\begin{equation*}
		(\psi^c)^G(y) = \frac{1}{|K_{\xi^c}||D_{\xi^c}|}\sum_{k \in K_{\xi^c}} \sum_{d \in \Omega} \psi^c(kdyd^{-1}k^{-1}) = \frac{1}{|D_{\xi^c}|}\sum_{d \in \Omega} \psi^c(dyd^{-1}),
	\end{equation*}
	where the last equality used that $\psi^c$ is $K_{\xi^c}$-invariant, as it is defined on $G_{\xi^c}$. 
	
	Now, if $d \in \Omega$, then $dyd^{-1} \in H_{\xi^c}$ is a $p$-element such that, since $\langle yL \rangle {\triangleleft}{\triangleleft} H/L$, then $\langle dyd^{-1}L \rangle {\triangleleft}{\triangleleft} H/L$, as $d \in H$. In particular $\langle dyd^{-1}L \rangle {\triangleleft}{\triangleleft} H_{\xi^c}/L$. Furthermore, since $\Sub_{G/L}(yL) = H/L$, then $\Sub_{G/L}(dyd^{-1}L) = H/L$, meaning $\Sub_{G_{\xi^c}/L}(dyd^{-1}L) = H_{\xi^c}/L$. By the properties of $\hat{F}$, it follows that $\psi^c(dyd^{-1}) = \epsilon \hat{F}(\psi^c)(dyd^{-1})$. Substituting above, we get
	\begin{equation*}
		\psi^G(y) = (\psi^c)^G(y) = \frac{\epsilon}{|D_{\xi^c}|}\sum_{d \in \Omega} \hat{F}(\psi^c)(dyd^{-1}) = \epsilon \hat{F}(\psi^c)^H(y) = F(\psi^G)(y),
	\end{equation*}
	as we wanted to show.
	
	Finally, write $\theta_{K_\xi} = \theta_\xi + R$ and $\varphi_{D_\xi} = \varphi_\xi + S$, where no irreducible constituent of $R, S$ lies over $\xi$, by the Clifford correspondence. Write also $\theta_D = e\varphi + pA + B$, where no irreducible constituent of $B$ is $\langle x \rangle$-invariant as in \Cref{relativeCorrespondence}. Then, we have
	\begin{equation*}
		(\theta_\xi)_{D_\xi} + R_{D_\xi} = e\varphi_\xi + pA_{D_\xi} + B_{D_\xi} + S.
	\end{equation*}
	Since no irreducible constituent of $R$ lies over $\xi$, while $\varphi_\xi$ does, it follows that
	\begin{equation*}
		(\theta_\xi)_{D_\xi} = e\varphi_\xi + pA' + B',
	\end{equation*}
	for some characters $A', B'$ such that all irreducible constituents of $B'$ lie under $B$. As $[\varphi, B] = 0$, we then get $[(\theta_\xi)_{D_\xi}, \varphi_\xi] \equiv e \equiv [\theta_D, \varphi] \pmod{p}$, as desired.
\end{proof}

Finally, we remove the condition that $\theta$ be $G$-invariant to prove \Cref{keyResult}. 

\begin{proof}[Proof of \Cref{keyResult}]
	Let $T = G_\theta$. Since $G = KH$, we have $T = KH_\theta$. Also, notice how $H_\theta/L = H/L \cap T/L = \Sub_{G_\theta/L}(xL)$, by \Cref{SubnormalizerProperties}. Hence, by \Cref{Part3}, there exists a bijection $\tilde{F}: \Irr(T\mid \theta) \to \Irr(H_\theta \mid \varphi)$ such that $\tilde{F}(\chi)(1)_p = \chi(1)_p$ and that $\chi(y) = \epsilon \tilde{F}(\chi)(y)$ for all $p$-elements $y \in T$ such that $\langle yL \rangle {\triangleleft}{\triangleleft} \Sub_{T/L}(yL)$ and $\Sub_{T/L}(yL) = H_\theta/L$. 
	
	Since $H_\theta = H_\varphi$ by \Cref{inertiaGroupLemma} we also have bijections $\Irr(T \mid \theta) \to \Irr(G \mid \theta)$ and $\Irr(H_\theta \mid \varphi) \to \Irr(H \mid \varphi)$ given by induction of characters. Composing this with $\tilde{F}$, we get a bijection $F(\psi^G) = (\tilde{F}(\psi))^H$. We claim this is the bijection we want. 
	
	First, $G = TH$, since $K \subseteq T$. Then, $|G:T| = |H:H_\theta|$ and we have
	\begin{equation*}
		F(\psi^G)(1)_p = \frac{|H:H_\theta|_p}{|G:T|_p}\psi^G(1)_p = \psi^G(1)_p,
	\end{equation*}
	as wanted. 
	
	Now let $y \in G$ be a $p$-element such that $\langle yL \rangle {\triangleleft}{\triangleleft} \Sub_{G/L}(yL)$ and $\Sub_{G/L}(yL) = H/L$. In particular, $y \in H$. We have $\psi^G(y) = (\psi^G)_H(y) = (\psi_{H_\theta})^H(y)$, by \cite[Problem 5.2]{CTFG}. By the induction formula, this gives us
	\begin{equation*}
		\psi^G(y) = \frac{1}{|H_\theta|} \sum_{\substack{h \in H \\ hyh^{-1} \in H_\theta}} \psi(hyh^{-1}).
	\end{equation*}
	
	Let $\Delta = \{h \in H \mid hyh^{-1} \in H_\theta\}$. Notice that, if $h \in \Delta$, then $hyh^{-1}$ is a $p$-element of $H_\theta$ such that $\langle hyh^{-1}L \rangle {\triangleleft}{\triangleleft} H/L$, since $h \in H$. Then, $\langle hyh^{-1}L \rangle {\triangleleft}{\triangleleft} H_{\theta}/L$. Also, $\Sub_{G/L}(hyh^{-1}L) = H/L$, again using $h \in H$. Consequently, $\Sub_{G_\theta/L}(hyh^{-1}L) = H_\theta/L$. Hence, by the properties of $\tilde{F}$, we have $\psi(hyh^{-1}) = \epsilon\tilde{F}(\psi)(hyh^{-1})$. Substituting above, we finally get
	\begin{equation*}
		\psi^G(y) = \frac{1}{|H_\theta|}\sum_{h \in \Delta} \epsilon\tilde{F}(\psi)(hyh^{-1}) = \epsilon \tilde{F}(\psi)^H(y) = \epsilon F(\psi^G)(y),
	\end{equation*}
	as desired. The fact that $\epsilon$ is the one we want follows from a similar argument to the end of \Cref{Part3}. This finishes the proof of \Cref{keyResult} and, with it, the proof of Theorem A is complete.
\end{proof}

In the proof of Theorem A, we needed $p \neq 2$ for two reasons. First, we needed it to ensure that the sign $\epsilon$ was consistent when applying the Clifford correspondence, even in the case where $G$ had a normal $p$-complement. We also needed it to ensure that the canonical extensions would map to one another when applying \Cref{TurullLemma}.

As for the condition $\langle x \rangle {\triangleleft}{\triangleleft} \Sub_G(x)$, it was used to ensure that the inertia groups of $\theta$ and $\varphi$ coincided in many different situations, through the use of \Cref{inertiaGroupLemma}. At the time of writing, it is not clear if it is possible to get an analogous result without this condition. It was also used, as mentioned before, to apply \Cref{TurullLemma} in \Cref{Part1}, by noticing that $\Sub_G(x)$ would be, in fact, a normalizer in that case.

We note that it is not too difficult to find non-picky $p$-elements satisfying the condition $\langle x \rangle {\triangleleft}{\triangleleft} \Sub_G(x)$. To that end, we present two constructions. 

Let $G$ be a finite group such that $x \in G$ is a picky $p$-element and let $H$ be a finite group which does not have a normal Sylow $p$-subgroup. Then, $(x, 1) \in G \times H$ is a $p$-element and, by \cite[Lemma 2.13]{GunterSubnormalizer}, $\Sub_{G \times H}(x, 1) = \Sub_G(x) \times H$. Since $x$ is picky, $\langle (x, 1) \rangle$ is subnormal in this subnormalizer and it is not picky, since $H$ does not have a normal Sylow $p$-subgroup.

For our second construction, we need a lemma.

\begin{lem}
	Let $x \in G$ be a $p$-element such that $S_G(x) \leq G$ and let $q$ be a prime. Then, in the group $\Gamma = G \wr C_q$, $S_\Gamma(x, \ldots, x) \leq \Gamma$.
\end{lem}

\begin{proof}
	Since $S_G(x) \leq G$, we have $S_G(x) = \NB_G(Q)$, where $Q$ is the intersection of all Sylow $p$-subgroups containing $x$, by \Cref{SubnormalizerProperties}. Let $P \in \Syl_p(G)$ contain $x$ and write $\Gamma = \langle a \rangle B$, where $B \cong G^q$ is the base group and $\langle a \rangle \cong C_q$ acts by shifting coordinates to the right. 
	
	First, let us assume that $q \neq p$. Then, $P \times \cdots \times P$ is a Sylow $p$-subgroup of $\Gamma$, whose normalizer is easily seen to be $\langle a \rangle(\NB_G(P) \times \cdots \times \NB_G(P))$. Suppose $y = a^i(g_1, \ldots, g_q)$ is such that $(x, \ldots, x) \in (P \times \cdots \times P)^y$. Then, $x \in P^{g_i}$ for all $i$. By \cite[Lemma 2.9]{GunterSubnormalizer}, for instance, $g_i \in \Sub_G(x) = \NB_G(Q)$ for all $i$. 
	
	Now, we have that $\NB_\Gamma((P \times \cdots \times P)^y) = \langle a^y \rangle(\NB_G(P^{g_1}) \times \cdots \times \NB_G(P^{g_q}))$. Notice, by direct calculation, that $a^y = a(g_q^{-1}g_1, \ldots, g_{q-1}^{-1}g_q)$. By our hypothesis, since $\Sub_G(x) = \NB_G(Q)$ and $x \in P^{g_i}$, then, for each $i$, $\NB_G(P^{g_i}) \leq \NB_G(Q)$ by \Cref{SubnormalizerProperties} (i). At the same time, $a(g_q^{-1}g_1, \ldots, g_{q-1}^{-1}g_q) \in \langle a \rangle (\NB_G(Q) \times \cdots \times \NB_G(Q))$. Hence, we get
	\begin{equation*}
		\NB_\Gamma((P \times \cdots \times P)^y) \leq \langle a \rangle (\NB_G(Q) \times \cdots \times \NB_G(Q)),
	\end{equation*}
	where $y$ was an arbitrary element such that $(x, \ldots, x) \in (P \times \cdots \times P)^y$. Since $(x, \ldots, x)$ is clearly subnormal in $\langle a \rangle (\NB_G(Q) \times \cdots \times \NB_G(Q))$, by \cite[Proposition 2.6]{GunterSubnormalizer}, $S_{\Gamma}(x, \ldots, x)$ is a subgroup of $\Gamma$, as wanted.
	
	Now, assume that $q = p$. In this case, $\langle a \rangle(P \times \cdots \times P)$ is a Sylow $p$-subgroup of $\Gamma$. Direct computation yields that its normalizer in $\Gamma$ is contained in the subgroup $\langle a \rangle(\NB_G(P) \times \cdots \times \NB_G(P))$.
	
	Once again, let $y = a^i (g_1, \ldots, g_p)$ be such that $(x, \ldots, x) \in (\langle a \rangle(P \times \cdots \times P))^y$. Then, we have 
	\begin{align*}
		(x, \ldots, x) \in B \cap (\langle a^y \rangle(P^{g_1} \times \cdots \times P^{g_p})) &= (P^{g_1} \times \cdots \times P^{g_p}) (B \cap \langle a^y \rangle) \\
		&= P^{g_1} \times \cdots \times P^{g_p},
	\end{align*}
	using that $a^y = a^{(g_1, \ldots, g_p)}$ and $(g_1, \ldots, g_p) \in B$. This then implies $x \in P^{g_i}$ for all $i$, which, once more, means $g_i \in \Sub_G(x) = \NB_G(Q)$. 
	
	Just as before, $\NB_G(P^{g_i}) \leq \NB_G(Q)$ by \Cref{SubnormalizerProperties} (i). It follows that the normalizer in $\Gamma$ of $(\langle a \rangle (P \times \cdots \times P))^y$ is contained in $\langle a \rangle(\NB_G(Q) \times \cdots \times \NB_G(Q))$. Since $(x, \ldots, x)$ is subnormal in this latter group, we conclude once again that $S_\Gamma(x, \ldots, x) \leq \Gamma$, using \cite[Proposition 2.6]{GunterSubnormalizer}.
\end{proof}

So let $G$ be any $p$-solvable group containing a $p$-element $x$ such that $S_G(x) \leq G$ (for example, using the previous construction). Consider $\Gamma = G \wr C_q \wr C_p$, where $q \neq p$ is a prime. By \cite[Remark 4.7]{HuppertBlackburn}, for example, $\Gamma$ has $p$-length one more than that of $G$. At the same time, by the previous lemma applied twice, the subnormalizer of $(x, \ldots, x)$ is a subgroup of $\Gamma$.

\begin{ack}
	The author would like to express his utmost gratitude to Alexander Moretó for sharing notes on the conjectures and to Noelia Rizo for invaluable help in proofreading and discussing the contents of the article.
\end{ack}

% Bibliography
%%%%%%%%%%%%%%%%%%%%%%%%%%%%%%%%%%%%%%%%%%%%%%%%%%%%%%%%%%%%%

%\bibitem{PickyManuscript} {\sc A. Moretó; N. Rizo}, Local representation theory, picky elements and subnormalizers, \emph{Unpublished manuscript} (2025);
%\printbibliography

\end{document}